\documentclass[reqno,a4paper,11pt]{article}
\usepackage{color}
\usepackage[latin1]{inputenc}
\usepackage[english]{babel}
\usepackage{amsmath,amssymb,amsfonts,amsthm,enumerate}
\usepackage{mathrsfs}
\usepackage[T1]{fontenc}
\usepackage[active]{srcltx}
\usepackage{epsfig}
\usepackage{graphicx}
\usepackage[
color,
final] % per togliere showkeys sostituire draft con final
{showkeys}

\definecolor{refkeybis}{gray}{.65}% per avere le labels stampate chiare
\definecolor{labelkeybis}{gray}{.65}% basta modificare .50:= intensita grigio
{\makeatletter
\def\SK@refcolor{\color{refkeybis}}%
\def\SK@labelcolor{\color{labelkeybis}}}

\numberwithin{equation}{section} % per cambiare la numerazione
\newtheorem{theorem}{Theorem}[section]
\newtheorem{lemma}[theorem]{Lemma}

\newtheorem{remark}[theorem]{Remark}
\newtheorem{proposition}[theorem]{Proposition}
\newtheorem{corollary}[theorem]{Corollary}

\newcommand{\R}{\mathbb{R}}

\newcommand{\K}{\mathcal{K}}
\newcommand{\LL}{\mathcal{L}}

\newcommand{\<}{\langle}
\renewcommand{\>}{\rangle}
\renewcommand{\a}{\alpha}
\renewcommand{\b}{\beta}
\renewcommand{\d}{\delta}
\newcommand{\D}{\Delta}
\newcommand{\e}{\varepsilon}
\newcommand{\g}{\gamma}
\newcommand{\G}{\Gamma}
\renewcommand{\k}{\kappa}
\renewcommand{\l}{\lambda}
\renewcommand{\L}{\Lambda}
\newcommand{\n}{\nabla}
\newcommand{\var}{\varphi}

\renewcommand{\S}{\mathbb{S}}

\newcommand{\p}{\partial}

\renewcommand\div{\operatorname{div}}

\newcommand{\supp}{\operatorname{supp}}

 \newcommand{\tauV}{{\kern-3pt\tau}}

 \newcommand{\oVVVk}{\overline{\mbox{\boldmath$V$}}\kern-3pt}
 \newcommand{\tVVVk}{\tilde{\mbox{\boldmath$V$}}\kern-3pt}
 \newcommand{\Df}{(-\Delta)^{s}}
 
 \newcommand{\Lip}{{\rm Lip}}
 \newcommand{\logLip}{{\rm logLip}}
\title{Regularity of solutions\\ to the parabolic fractional obstacle problem}
\date{}
\author{Luis Caffarelli\thanks{Department of Mathematics,
The University of Texas at Austin, 1 University Station C1200, Austin TX 78712, USA.
E-mail: \texttt{caffarel@math.utexas.edu}} \,and Alessio Figalli\thanks{Department of Mathematics,
The University of Texas at Austin, 1 University Station C1200, Austin TX 78712, USA.
E-mail: \texttt{figalli@math.utexas.edu}}}

\begin{document}

\maketitle

\section{Introduction}

In recent years, there has been an increasing interest in studying constrained variational problems with a fractional
diffusion. 
One of the motivations comes from mathematical finance:
jump-diffusion processes where incorporated by Merton \cite{merton} into the theory of option evaluation
to introduce discontinuous paths in the dynamics of the stock's prices, in contrast with the classical lognormal diffusion model of Black and Scholes \cite{blsc}.
These models allow to take into account large price changes, and they have become increasingly popular for modeling market fluctuations,
both for risk management and option pricing purposes.

Let us recall that
an American option gives its holder the right to buy a stock at a given price prior (but not later) than a given time $T>0$.
If $v(\tau,x)$ represents the rational price of an American option with a payoff $\psi$ at time $T>0$, then $v$
will solve (in the viscosity sense) the following obstacle problem:
$$
\left\{
\begin{array}{l}
\min\{ \LL v,v-\psi\}=0,\\
v(T)=\psi.
\end{array}
\right.
$$
Here $\LL v$ is a (backward) parabolic integro-differential operator of the form
\begin{multline*}
\LL v= -v_\tau  - r v + \sum_{i=1}^n (r-d_i) x_i v_{x_i} - \frac{1}{2}\sum_{i,j=1}^n x_i x_j \sigma_{ij}  v_{x_i x_j}\\
 - \int \Bigl[v\bigl(\tau,x_1e^{y_1},\ldots,x_ne^{y_n}\bigr) - v(\tau,x) - \sum_{i=1}^n(e^{y_i} - 1) x_i v_{x_i}(\tau,x)\Bigr]\,\mu(dy),
\end{multline*}
where $r>0$, $d_i \in \R$, $\sigma=(\sigma_{ij})$ is a non-negative definite matrix,
and $\mu$ is a jump measure.
(We refer to the book \cite{conttankov} for an explanation of these models and more references.)
When the matrix $\sigma$ is uniformly elliptic, after the change of variable
$x_i \mapsto \log (x_i)$
the equation becomes uniformly parabolic (backward in time) and the diffusion part dominates. In particular, if no jump part is present (i.e., $\mu\equiv 0$),
then the regularity theory is pretty well-understood (see, for instance, \cite{lawsal}).

Here we assume that there is no diffusion (i.e., $\sigma\equiv 0$), so all the regularity should come from the jump part.
We also assume that the jump part behaves, at least at the leading order, as a fractional power of the Laplacian, so that the equation takes the form
\begin{equation}
\label{eq:tilde LL}
\LL v= -v_\tau  - r v - b\cdot \nabla u + \Df v + \K v,\qquad s\in (0,1),
\end{equation}
where $b=(d_1-r,\ldots,d_n-r)$, and $\K v$ is a non-local operator of lower order with respect to $\Df v$.

We now observe that the choice of $s \in (0,1)$ plays a key role:
\begin{enumerate}
 \item[-] $s>1/2$: In this case $\Df v$ is the leading term, so the regularity theory for solutions to \eqref{eq:tilde LL}
is expected to be the same one as that for the equation
\begin{equation}
\label{eq:frac heat visc back}
\left\{
\begin{array}{l}
\min\{- v_\tau + \Df v,v-\psi\}=0\quad \text{on }[0,T]\times \R^n,\\
v(T)=\psi\quad \text{on }\R^n.
\end{array}
\right.
\end{equation}
\item[-] $s\leq 1/2$: If $s<1/2$ then the leading term becomes $b\cdot \nabla v$, and we do not expect to have a
regularity theory for \eqref{eq:tilde LL}. On the other hand, in the borderline case $s=1/2$ one may expect some regularity due to the interplay between
$b\cdot \nabla v$ and $ - \Df v$ (but this becomes
a very delicate issue). However, when $b\equiv 0$, even if the diffusion term is of lower order with respect to the time derivative,
the equation is still parabolic and one may hope to prove some regularity for all values of $s$.
\end{enumerate}
The goal of this paper is to investigate the regularity theory for the model equation \eqref{eq:frac heat visc back}.
The reason for this is three-fold: first of all, considering this model case allows to avoid technicalities which may obscure the main ideas
behind the regularity theory that we will develop.
Moreover, since there is no transport term inside the equation, we are able to prove that solutions are as smooth as in the elliptic case
\cite{cafsalsil} \textit{for all} values of $s \in (0,1)$. Hence,
although when $s <1/2$ the time derivative is of higher order with respect to the elliptic part $\Df v$,
the regularity of solutions is as good as in the stationary case.
Finally, as described in Section \ref{sect:general eq}, once the general regularity theory for solutions of \eqref{eq:frac heat visc back} is established,
the adaptation of these proofs to the more general case \eqref{eq:tilde LL}
when $s>1/2$ should not present any major difficulty.\\

Let us remark that the fact that the smoothness of solutions of \eqref{eq:frac heat visc back} is the same as in the elliptic case
may look surprising. Indeed, the optimal regularity for the stationary problem
$\min\{ \Df v,v-\psi\}=0$ is $C_{x}^{1+s}(\R^n)$ \cite{athcaf,silv,cafsalsil}. On the other hand,
as we will show in Remark \ref{rmk:C1},
for any $\b\in(0,1)$ one can find a traveling wave solution to the parabolic obstacle problem $\min\{- v_\tau + (-\Delta)^{1/2} v,v-\psi\}=0$
which is $C^{1+\b}$ both in space and time, but not $C^{1+\g}$ for any $\g>\b$.
Hence, in order to prove that solutions to \eqref{eq:frac heat visc back} are $C^{1+s}$ in space, one has to exploit the crucial fact that
$v$ coincides with the obstacle at time $T$.

\section{Description of the results and structure of the paper}
In this section we introduce more in detail the problem, and describe our main result.

Let us observe that, by performing the change of variable $t=T-\tau$, all equations introduced in the previous section become forward in time.
From now on, we will always work with $t$ in place of $\tau$, so the payoff $\psi$ becomes the initial condition at time $0$.
%Moreover, although before we have written our equations using the viscosity formulation,
%here we will formulate the problem as a variational inequality, as this is more convenient for our purposes.

\subsection{Preliminary definitions}
\label{subsect:prelim}
The fractional Laplacian can be defined as
$$
\Df f := \widehat{|\xi|^{2s}\hat f}\qquad \forall\,f \in C^\infty_c(\R^n),
$$
so that
$\int f \Df g =\< f,g\>_{\dot H^s}$. There are also two other different ways to define $\Df$.
The first one is through an integral kernel: there exists a positive constant $C_{n,s}$ such that
$$
-\Df f(x)=C_{n,s}\int \frac{f(x')-f(x)}{|x'-x|^{n+2s}}\,dx' \qquad \forall\, f \in C^\infty_c(\R^n),
$$
where the integral has to be intended in the principal value sense. (This can be proved, for instance, by computing the Fourier transform of $|\xi|^{2s}$.)
The second one is through a Dirichlet-to-Neumann operator, as shown in \cite{cafsil}: 
given $a \in (-1,1)$, for any function $f\in C^\infty_c(\R^n)$ denote by $F:\R^n\times \R^+ \to \R$ the $L_a$-harmonic extension of $f$,
i.e.,
$$
\left\{
\begin{array}{ll}
L_a F(x,y):=\div_{x,y}\bigl(y^a \nabla_{x,y} F(x,y)\bigr)=0&\text{on $\R^n\times\R^+$},\\
F(x,0)=f(x)&\text{on $\R^n$}.
\end{array}
\right.
$$
Then there exists a positive constant $c_{n,s}$ such that
$$
\lim_{y\to 0^+}y^a F_y(x,y)=-c_{n,s}\Df f(x,0),\qquad s:=\frac{1-a}{2} \in (0,1).
$$
In the sequel, we will make use of all of the three above characterizations of the fractional Laplacian.
However, in order to simplify the notation, we will conventionally assume that $C_{n,s}=c_{n,s}=1$, so that
\begin{equation}
\label{eq:equivalent frac lapl}
-\Df f=\int \frac{f(x')-f(x)}{|x'-x|^{n+2s}}\,dx'=\lim_{y\to 0^+}y^a F_y(x,y).
\end{equation}

We will also need the notion of semiconvex function:
a function $w:\R^n\to\R$ is said to be $C$-semiconvex for some constant $C \in \R$ if $w+C|x|^2/2$ is convex.\\

Finally, to measure the regularity of the solutions we will use space-time H\"older, Lipschitz, and logLipschitz spaces:
given $\a,\b,\g,\d \in (0,1)$, and $[a,b] \subset \R$, we say that:\\
$w \in C_{t,x}^{\a,\b} ([a,b]  \times \R^n)$ if
\begin{align*}
\|w\|_{C_{t,x}^{\a,\b}([a,b]  \times \R^n)}&:=\|w\|_{L^\infty({[a,b]  \times \R^n})}+[w]_{C_{t,x}^{\a,\b}([a,b]  \times \R^n)}\\
&= \|w\|_{L^\infty({[a,b]  \times \R^n})}+\sup_{[a,b]  \times \R^n} \frac{|w(t,x)-w(t',x')|}{|t-t'|^\a +|x-x'|^\b} <+\infty;
\end{align*}
$w \in \Lip_tC_{x}^\b ([a,b]  \times \R^n)$ if
$$
\|w\|_{\Lip_tC_{x}^\b([a,b]  \times \R^n)}:=\|w\|_{L^\infty({[a,b]  \times \R^n})}+ \sup_{[a,b]  \times \R^n} \frac{|w(t,x)-w(t',x')|}{|t-t'| +|x-x'|^\b} <+\infty;
$$
$w \in \logLip_tC_x^\b ([a,b]  \times \R^n)$ if
$$
\|w\|_{\logLip_tC_{x}^\b([a,b]  \times \R^n)}:=\|w\|_{L^\infty({[a,b]  \times \R^n})}+ \sup_{[a,b]  \times \R^n} \frac{|w(t,x)-w(t',x')|}{|t-t'|\bigl(1+\bigl|\log|t-t'|\bigr|\bigr) +|x-x'|^\b} <+\infty.
$$
We will also use the notation $w \in C_{t,x}^{\a-0^+,\b} ([a,b] \times \R^n)$ if
$$
w \in C_{t,x}^{\a-\e,\b} ([a,b]  \times \R^n) \qquad \forall\,\e>0,
$$
and $w \in C_{t,x}^{\a,\b} ((a,b]  \times \R^n)$ if
$$
w \in C_{t,x}^{\a,\b} ([a+\e,b]  \times \R^n) \qquad \forall\,\e>0
$$
(analogous definitions hold for the other spaces).

\subsection{The main result}
Let $\psi:\R^n \to \R^+$ be a globally Lipschitz function of class $C^2$ satisfying $\int_{\R^n}\frac{|\psi|}{(1+|x|)^{n+2s}}<+\infty$
and $\Df \psi \in L^\infty(\R^n)$.
Fix $s \in (0,1)$, and let  $u:[0,T]\times \R^n \to \R$ be a (continuous) viscosity solution to the obstacle problem
\begin{equation}
\label{eq:frac heat}
\left\{
\begin{array}{l}
\min\{ u_t + \Df u,u-\psi\}=0 \quad \text{on }[0,T]\times \R^n,\\
u(0)=\psi\quad \text{on } \R^n.
\end{array}
\right.
\end{equation}
Existence and uniqueness of such a solution
follows by standard results on obstacle problems\footnote{
Here, existence of solutions is not the main issue: for instance,
one can construct solutions by using probabilistic formulas involving stochastic processes
and stopping times \cite{conttankov}.
Another possibility is to approximate the equation
using a penalization method (as done in the proof of Lemma \ref{lemma:comparison})
and then use the a priori bounds on the approximate solutions (see the proofs of Lemmas \ref{lemma:basic} and \ref{lemma:Linfty bound frac heat}) to show existence by compactness.
The fact that these two notions of solutions (the probabilistic one and the one constructed by approximation) coincide, follows from standard
comparisons principle for viscosity solutions.}.
The main goal of this paper is to investigate the smoothness
of solutions to the above equations, planning 
to address in a future work the regularity of the free boundary.

Our main result is the following:
\begin{theorem}
\label{thm:main}
Assume that $\psi \in C^2(\R^n)$, with
$$
\|\nabla \psi\|_{L^\infty(\R^n)}+\|D^2 \psi\|_{L^\infty(\R^n)}+\|\Df \psi\|_{C_x^{1-s}(\R^n)} <+\infty,
$$
and let $u$ be the unique continuous viscosity solution of \eqref{eq:frac heat}.
Then $u$ is globally Lipschitz in space-time on $[0,T]\times \R^n$, and satisfies
\begin{equation}
\label{eq:opt reg}
\left\{
\begin{array}{lll}
%u_t \in \Lip_tC_x^{1-s}((0,T]\times\R^n),&\Df u \in\Lip_tC_x^{1-s}((0,T]\times \R^n)&\text{if }s<1/3;\\
u_t \in \logLip_tC_x^{1-s}((0,T]\times\R^n),&\Df u \in\logLip_tC_x^{1-s}((0,T]\times \R^n)&\text{if }s\leq 1/3;\\
u_t \in C_{t,x}^{\frac{1-s}{2s}-0^+,1-s}((0,T]\times \R^n),&\Df u \in C_{t,x}^{\frac{1-s}{2s},1-s}((0,T]\times \R^n)&\text{if }s>1/3.\\
\end{array}
\right.
\end{equation}
\end{theorem}
Let us make some comments.
First of all we recall that, for the stationary version of the obstacle problem, solutions belong to $C_{x}^{1+s}(\R^n)$
(or equivalently, $\Df u \in C_{x}^{1-s}(\R^n)$),
and such a regularity result is optimal \cite{silv,cafsalsil}.
Hence, at least concerning the spatial regularity, our result is optimal, too.

Once the $C_x^{1-s}$-regularity of $\Df u$ is established, the fact that $s=1/3$ plays a special role is not surprising: indeed,
the operator $\p_t+\Df$ is invariant under the scaling $(t,x)\mapsto (\l^{2s}t,\l x)$. Hence, a spatial regularity $C^{1-s}_x$ naturally corresponds to a time regularity
$C^{\frac{1-s}{2s}}_t$, provided $\frac{1-s}{2s}< 1$, that is, $s > 1/3$ (see \eqref{eq:infty regularity holder 1}-\eqref{eq:infty regularity holder 2}
in the Appendix).

Finally, concerning the regularity in time, when $s=1/2$ one can construct traveling wave solutions
which are $C^{1+1/2}$ both in space and time, see Remark \ref{rmk:C1}.
Hence our result is almost optimal in time, at least when $s=1/2$ (the result would be optimal if we did not have the $0^+$
in the H\"older exponent). Moreover, the regularity in time is almost optimal also in the limit
$s \to 1$ (since, when $s=1$, it is well-known that solutions are $C^1$ in time and $C^{1,1}$ in space \cite{brekind,caffActa,cafffried-contStefan}).
Hence, it may be expected that our result is almost optimal in time for all $s \in (0,1)$ (or at least for $s> 1/3$).

\subsection{Structure of the paper}

The paper is structured as follows:
first, in Section \ref{sect:prelim} we discuss some basic properties of solutions of \eqref{eq:frac heat},
like the validity of a comparison principle, the Lipschitz regularity in space-time, the semiconvexity in space, and
the boundedness of $\Df u$. Moreover, we will show that solutions are $C^1$ for $s \geq 1/2$,
and, as explained in Remark \ref{rmk:C1},
$C^1$-regularity in space is optimal when $s=1/2$ unless one exploits the additional information that the solution coincides with the obstacle at the initial time.

In Section \ref{sect:opt reg}, we first use an iteration method to show that, for any $t>0$, $\Df u(t)$ is $C_{x}^{\a}$ near any free boundary point (Subsection \ref{sect:general C1a}).
Then, we prove a monotonicity formula which allows to show that $\Df u(t)$ is $C_{x}^{1-s}$ near any free boundary point for all $t>0$  (Subsection \ref{subsect:monot}).
Finally, combining the fact that $\Df u(t)$ is $C_{x}^{1-s}$ on the contact set with equation \eqref{eq:frac heat}, a bootstrap argument 
allows to prove Theorem \ref{thm:main} (Subsection \ref{sect:C1s u}).

In Section \ref{sect:general eq} we briefly describe what are the main modifications to perform in order to extend the regularity result in Theorem \ref{thm:main}
to solutions of \eqref{eq:tilde LL} when $s>1/2$, leaving the details to some future work.

Finally, in the appendix we collect some regularity properties of the fractional heat operator $\p_t+\Df$.

\section{Basic properties of solutions}\label{sect:prelim}

Here we discuss some elementary properties of solutions of \eqref{eq:frac heat}.
Actually, since many of them do not rely on the fact that $u$ coincides with the obstacle at time $0$,
we consider solutions to
\begin{equation}
\label{eq:frac heat no initial}
\left\{
\begin{array}{ll}
\min\{u_t + \Df u,u-\psi\}=0\quad \text{on }[0,T]\times \R^n,\\
u(0)=u_0\quad \text{on } \R^n,
\end{array}
\right.
\end{equation}
where $u_0 \geq \psi$ is a globally Lipschitz semiconvex function.
Most of the properties of $u$ will be a consequence of the following general comparison principle:

\begin{lemma}[Comparison principle]
\label{lemma:comparison}
Let
$\psi,\tilde\psi:\R^n\to \R$ be two continuous functions, and
assume that $u,\tilde u:[0,T]\times \R^n\to \R$ are viscosity solutions of 
\begin{equation}
\label{eq:gen frac heat 1}
\left\{
\begin{array}{l}
\min\{ u_t + \Df u,u-\psi\}=0\quad \text{on }[0,T]\times \R^n,\\
u(0)=u_0\quad \text{on }\R^n,
\end{array}
\right.
\end{equation}
and
\begin{equation}
\label{eq:gen frac heat 2}
\left\{
\begin{array}{l}
\min\{ \tilde u_t + \Df \tilde u,\tilde u-\tilde \psi\}=0\quad \text{on }[0,T]\times \R^n,\\
\tilde u(0)=\tilde u_0\quad \text{on }\R^n,
\end{array}
\right.
\end{equation}
respectively.
Assume that $u_0\leq \tilde u_0$ and $\psi \leq \tilde\psi$.
Then $u(t) \leq \tilde u(t)$ for all $t \in [0,T]$. 
\end{lemma}
\begin{proof}
We use a penalization method: it is well-known that solutions
of \eqref{eq:frac heat no initial} can be constructed as
a limit of $u^\e$ as $\e\to 0$, where $u^\e$ is smooth solutions of
\begin{equation}
\label{eq:approx frac heat}
\left\{
\begin{array}{ll}
u_t^\e +\Df u^\e = \beta_\e(u^\e - \psi_\e)&\text{on }[0,T]\times \R^n\\
u^\e(0)=u_0^\e \geq \psi_\e\quad \text{on } \R^n,,
\end{array}
\right.
\end{equation}
with $u_0^\e,\psi_\e \in C^\infty_c(\R^n)$, $\beta_\e(s)=e^{-s/\e}$, $\psi_\e \to \psi$, $\Df\psi_\e \to \Df\psi$, and $u_0^\e\to u_0$
locally uniformly as $\e\to 0$ (see for instance \cite[Chapter 3]{chipotbook} for a proof in the classical parabolic case).

Hence, it suffices to prove the comparison principle at the level of the approximate equations, assuming $u^\e(0)\leq \tilde u^\e(0)$ and $\psi_\e\leq\tilde\psi_\e$.
Let us observe that, since $\psi_\e  \leq \tilde\psi_\e$ and $\beta_\e'\leq 0$, we have
$$
\beta_\e(\cdot - \psi_\e)\leq \beta_\e(\cdot - \tilde \psi_\e),
$$
which implies
$$
u_t^\e +\Df u^\e = \beta_\e(u^\e - \psi_\e)\qquad \text{on }[0,T]\times \R^n\\
$$
$$
\tilde u_t^\e +\Df \tilde u^\e=\beta_\e(\tilde u^\e - \tilde \psi_\e) \geq \beta_\e(\tilde u^\e - \psi_\e)\qquad\text{on }[0,T]\times \R^n.
$$
Since $u^\e(0)\leq \tilde u^\e(0)$, by standard comparison principle for parabolic equations (see for instance the argument in the proof 
of Lemma \ref{lemma:Linfty bound frac heat} below)
we get $u^\e \leq \tilde u^\e$, as desired.
\end{proof}

The following important properties are an immediate consequence of the above result:
\begin{lemma}
\label{lemma:basic}
Let $u$ be a solution of \eqref{eq:frac heat no initial}, and assume that $u_0$ and $\psi$ are globally Lipschitz and $C_0$-semiconvex. Then:
\begin{enumerate}
\item[(i)] $u(t)$ is Lipschitz for all $t \in [0,T]$, with $\|\nabla u(t)\|_{L^\infty(\R^n)} \leq \max\{\|\nabla u_0\|_{L^\infty(\R^n)},\|\nabla \psi\|_{L^\infty(\R^n)}\}$.
\item[(ii)] $u(t)$ is $C_0$-semiconvex for all $t \in [0,T]$.
\end{enumerate}
Moreover, if $u_0=\psi$ then
\begin{enumerate}
\item[(iii)] $[0,T]\ni t\mapsto u(t,x)$ is non-decreasing in time.
\end{enumerate}
\end{lemma}
\begin{proof}
(i) Observe that, for every vector $v\in\R^n$ and any constant $C\in\R$,
$u(t,x+v) + C|v|$ solves \eqref{eq:gen frac heat 1} starting from $u_0(x+v)+C|v|$ with obstacle
$\psi(x+v)+C|v|$.
Moreover, if $C:=\max\{\|\nabla u_0\|_{L^\infty(\R^n)},\|\nabla \psi\|_{L^\infty(\R^n)}\}$, then
$u_0(x+v)+C|v|\geq u_0(x)$ and $\psi(x+v)+C|v|\geq \psi(x)$. Hence, by Lemma \ref{lemma:comparison} we obtain
$$
u(t,x+v) + C|v| \geq u(t,x)\qquad \forall\, x,v \in\R^n,\,t\geq 0.
$$
The Lipschitz regularity of $u(t)$ follows.\\
(ii) As above, we just remark that $u(t,x+v)+u(t,x-v)+C|v|^2$ solves \eqref{eq:gen frac heat 1}
for every $C \in \R$. Hence, by choosing $C:=2C_0$ we get
$u_0(x+v) +u_0(x-v)+2C_0|v|^2 \geq 2u_0(x)$ and $\psi(x+v) +\psi(x-v)+2C_0|v|^2 \geq 2\psi(x)$, and we conclude as above using Lemma \ref{lemma:comparison}.\\
(iii) We observe that, for any $\e \geq 0$,
the function $u(t+\e,x)$ solves \eqref{eq:frac heat}
starting from  $u(\e,\cdot)$. Hence, since $u(\e,\cdot) \geq \psi$, by the comparison principle we obtain
$$
u(t+\e,x) \geq u(t,x)\qquad \forall \,t, \e\geq 0.
$$
\end{proof}

We now prove the following important bounds:
\begin{lemma}
\label{lemma:Linfty bound frac heat}
Let $u$ be a solution of \eqref{eq:frac heat no initial}. Then
\begin{equation}
\label{eq:Linfty bound frac heat}
0 \leq u_t + \Df u \leq \|\Df \psi\|_{L^\infty(\R^n)},
\end{equation}
\begin{equation}
\label{eq:Linfty bound ut}
\|u_t\|_{L^\infty([0,T]\times \R^n)}  \leq \|\Df u_0\|_{L^\infty(\R^n)} ,
\end{equation}
In particular $\Df u$ is bounded, with
\begin{equation}
\label{eq:Linfty bound Df u}
\|\Df u\|_{L^\infty([0,T]\times \R^n)} \leq \|\Df \psi\|_{L^\infty(\R^n)}+\|\Df u_0\|_{L^\infty(\R^n)}.
\end{equation}
\end{lemma}
\begin{proof}
As in the proof of Lemma \ref{lemma:comparison}, we use a penalization method: we consider solutions $u^\e$ to \eqref{eq:approx frac heat}, and
we prove a uniform (with respect to $\e$) $L^\infty$-bound on both 
$\beta_\e(u^\e - \psi_\e)$ and $u_t^\e$.

\textit{$\bullet$ $L^\infty$-bound on $\beta_\e(u^\e - \psi_\e)$.} Since $\beta_\e \geq 0$,
we only need an upper bound.

Assume that $\inf_{[0,T]\times \R^n}(u^\e - \psi_\e) <0$
(otherwise the problem is trivial), and let $\var$ be a smooth function which grows like  $|x|^s$ at infinity.
Then, since $u^\e$ vanishes at infinity (being a solution to a smooth parabolic equation
starting from a compactly supported initial datum),
for any $\delta>0$ small
we can consider $(t_\e^\delta,x_\e^\delta)$ a minimum point
for $u^\e - \psi_\e +\frac{\delta}{T-t}+\delta \var$ over $[0,T]\times\R^n$. Of course,
$\min_{[0,T]\times \R^n}\left(u^\e - \psi_\e+\frac{\delta}{T-t}+\delta \var\right) <0$ for $\delta$ sufficiently small,
which implies that $(t_\e^\delta,x_\e^\delta)$ belongs to the interior of $(0,T)\times \R^n$.
Hence
$$
u_t^\e(t_\e^\delta,x_\e^\delta) + \frac{\delta}{(T-t^\delta_\e)^2}=0,\qquad
\Df u^\e(t_\e^\delta,x_\e^\delta) - \Df \psi_\e(x_\e^\delta) +\delta  \Df \var(x_\e^\delta)\leq 0,
$$
which combined with \eqref{eq:approx frac heat} gives
$$
\beta_\e(u^\e - \psi_\e)(t_\e^\delta,x_\e^\delta) \leq
\Df \psi_\e(x_\e^\delta)-\frac{\delta}{(T-t^\delta_\e)^2}- \delta  \Df \var(x_\e^\delta)\leq \|\Df \psi_\e\|_{L^\infty(\R^n)}+O(\delta).
$$
Since $(u^\e - \psi_\e)(t_\e^\delta,x_\e^\delta)\to \inf_{[0,T]\times \R^n}(u^\e - \psi_\e)$ as $\d\to 0$ and $\beta_\e'\leq 0$ we obtain
$$
\sup_{[0,T]\times \R^n}\beta_\e(u^\e - \psi_\e) = \lim_{\d\to 0}\beta_\e (u^\e - \psi_\e)(t_\e^\delta,x_\e^\delta) \leq \|\Df \psi_\e\|_{L^\infty(\R^n)},
$$
so that \eqref{eq:Linfty bound frac heat} follows letting $\e\to 0$.

\textit{$\bullet$ $L^\infty$-bound on $u_t^\e$.} We use the same argument as in \cite[Lemma 2.1]{frikin}:
differentiating \eqref{eq:approx frac heat} with respect to $t$ we obtain that $w^\e:=u_t^\e$ solves
$$
\left\{
\begin{array}{ll}
w^\e_t +\Df w^\e = \beta_\e'(u^\e - \psi_\e)w^\e&\text{on }[0,T]\times \R^n\\
w^\e(0)=-\Df u^\e(0)\quad \text{on }\R^n.
\end{array}
\right.
$$
Since $\beta_\e' \leq 0$ and $\|w^\e(0)\|_{L^\infty}=\|\Df u_0^\e\|_{L^\infty}$, using a maximum principle argument (as above), we infer that 
$$
\|u_t^\e\|_{L^\infty([0,T]\times\R^n)}=\|w^\e\|_{L^\infty([0,T]\times\R^n)}
\leq \|w^\e(0)\|_{L^\infty(\R^n)}=\|\Df u_0^\e\|_{L^\infty(\R^n)}.
$$
Letting $\e\to 0$ we get \eqref{eq:Linfty bound ut}, as desired.
\end{proof}

The above result together with Lemma \ref{lemma:basic}(i)
gives the following:

\begin{corollary}[Lipschitz regularity in space-time]
\label{cor:Lip}
Let $u$ be a solution of \eqref{eq:frac heat no initial}. Then
$$
\|u_t\|_{L^\infty([0,T]\times \R^n)}+
\|\nabla u\|_{L^\infty([0,T]\times \R^n)}\leq
\max\{\|\nabla u_0\|_{L^\infty(\R^n)},\|\nabla \psi\|_{L^\infty(\R^n)}\}+
\|\Df u_0\|_{L^\infty(\R^n)}.
$$
\end{corollary}

In the sequel we will also need the following result:
\begin{lemma}
\label{lem:sign frac lapl}
Let $u$ be a solution of \eqref{eq:frac heat} with
$$
\|\nabla \psi\|_{L^\infty(\R^n)}+\|\Df \psi\|_{L^\infty(\R^n)}+\|\nabla u_0\|_{L^\infty(\R^n)}+
\|\Df u_0\|_{L^\infty(\R^n)}<+\infty,
$$
and fix $t_0>0$. Then
\begin{equation}
\label{eq:frac lapl leq 0}
0 \leq \Df u(t_0) \leq \|\Df \psi\|_{L^\infty(\R^n)} \qquad \text{for a.e. $x \in \{u(t_0)=\psi\}$},
\end{equation}
\begin{equation}
\label{eq:frac lapl geq 0}
\Df u(t_0) \leq 0 \qquad \text{on $\{u(t_0)>\psi\}$}.
\end{equation}
\end{lemma}
\begin{proof}
Let us recall that, thanks to Corollary \ref{cor:Lip}, $u$ is Lipschitz in time. So, Lemma  \ref{lemma:basic}(iii) gives $u_t\geq 0$ a.e.

Moreover, since $u_t=0$ a.e. on the contact set $\{u=\psi\}$, $u$ satisfies
$$
u_t +\Df u=0\quad \text{in }\{u>\psi\},
\qquad
u_t=0\quad \text{a.e. on }\{u=\psi\},
$$
which gives
$$
u_t +\Df u=\bigl(\Df u\bigr) \chi_{\{u=\psi\}}
$$
both in the almost everywhere sense and in the sense of distribution. (Observe that the above formula makes sense
since $\Df u$ is a bounded function, see Lemma \ref{lemma:Linfty bound frac heat}).
Hence, $u$ solves the smooth parabolic equation $u_t+\Df u =f$, with $f$ globally bounded and vanishing inside the open set $\{u>\psi\}$. So,
a simple application of Duhamel formula shows that $u$ is smooth inside $\{u>\psi\}$.
In particular, this fact combined with the non-negativity of $u_t$ implies
$$
\Df u(t_0)=-u_t(t_0) \leq 0 \qquad \text{on $\{u(t_0)>\psi\}$},
$$
that is, \eqref{eq:frac lapl geq 0}.

We now prove \eqref{eq:frac lapl leq 0}.
Since $u_t=0$ a.e.
on the contact set  $\{u=\psi\}$, \eqref{eq:Linfty bound frac heat}
gives
\begin{equation}
\label{eq:bound space-time}
0 \leq \Df u \leq \|\Df \psi\|_{L^\infty}\qquad \text{ for a.e. $(t,x) \in \{u=\psi\}$.}
\end{equation}
Now, to show that the bound $0 \leq \Df u(t_0) \leq \|\Df \psi\|_{L^\infty(\R^n)}$ holds a.e. on $\R^n$ \textit{for every $t_0 \in [0,T]$},
we observe that the map
$$
t \mapsto u(t) \in L^2_{{\rm loc}}(\R^n)
$$
is uniformly continuous
(this is a consequence of the Lipschitz continuity in time, see Corollary \ref{cor:Lip}), which together with the uniform bound \eqref{eq:Linfty bound frac heat}
implies that the map
$$
t \mapsto \Df u(t) \in L^2_{{\rm loc}}(\R^n)
$$
is weakly continuous. Thanks to this fact, we easily deduce the desired estimate.
Indeed, fix $\e>0$, $A\subset \{u(t_0)=\psi\}$ a bounded Borel set,
and test \eqref{eq:bound space-time} against the function $\chi_{[t_0-\e,t_0]}\chi_{A}$.
Since the sets $\{u(t)=\psi\}$ are decreasing in time (see Lemma \ref{lemma:basic}(iii)),
we have $[t_0-\e,t_0]\times A\subset \{u=\psi\}$, which together with \eqref{eq:bound space-time} gives
$$
0 \leq  \int_{t_0-\e}^{t_0}\int_{A} \Df u \leq \|\Df \psi\|_{L^\infty}\,|A|\,\e.
$$
Dividing by $\e$ and letting $\e \to 0$, by the weak-$L^2$ continuity of $t \mapsto \Df u(t)$ we deduce
$$
0 \leq \int_{A} \Df u(t_0) \leq \|\Df \psi\|_{L^\infty}\,|A|\qquad \forall\text{ $A\subset \{u(t_0)=\psi\}$ Borel bounded},
$$
so that the desired bound follows.
\end{proof}

We now show that the uniform semiconvexity of $u(t)$, together with the $L^\infty$-bound on $\Df u(t)$, implies that solutions are $C^1$ in space when $s \geq 1/2$
(actually, when $s>1/2$, by elliptic regularity theory the boundedness of $\Df u(t)$ implies that $u(t) \in C^{2s-0^+}_{\rm loc}(\R^n)$).
As we will show in Remark \ref{rmk:C1} below, unless the contact set shrinks in time, this regularity result is optimal for $s=1/2$.
\begin{proposition}[$C^1$-spatial regularity]
\label{prop:C1x reg}
Let $u$ be a solution of \eqref{eq:frac heat no initial} with $s \in [1/2,1)$.
Assume that $u_0$ and $\psi$ are semiconvex, and that $\|\Df u_0\|_{L^\infty(\R^n)}+\|\Df \psi\|_{L^\infty(\R^n)}<+\infty$.
Then $u(t) \in C^1(\R^n)$ for all $t \in [0,T]$.
Moreover the modulus of continuity of $\nabla u$ depends only on $s$,
the semiconvexity constant of $u_0$ and $\psi$, and on $\|\Df u_0\|_{L^\infty(\R^n)}+\|\Df \psi\|_{L^\infty(\R^n)}$.
\end{proposition}
\begin{proof}
First of all, we claim that, for every fixed $t \in [0,T]$, the map $x\mapsto -\Df u(t,x)$ is lower semicontinuous.
Indeed, recall that if $C_0$ denotes a semiconvexity constant for both $u_0$ and $\psi$, then $u(t)$ is $C_0$-semiconvex
for all $t \in [0,T]$ (see Lemma \ref{lemma:basic}(ii)).
Hence $\Df u(t,x)$ is pointwise defined \textit{at every} $x \in \R^n$,  and is given by (see \eqref{eq:equivalent frac lapl})
\begin{align*}
-\Df u (t,x)&=\int_{B_1} \frac{u(t,x+h)+u(t,x-h)-2u(t,x)}{2|h|^{n+2s}}\,dh + \int_{\R^n\setminus B_1}
\frac{u(t,x+h)-u(t,x)}{|h|^{n+2s}}\,dy\\
&=\int_{B_1} \frac{u(t,x+h)+u(t,x-h)-2u(t,x)+2C_0|h|^2}{2|h|^{n+2s}}\,dh\\
&\quad-C_0 C(n,s) +
\int_{\R^n\setminus B_1} \frac{u(t,x+h)-u(t,x)}{|h|^{n+2s}}\,dy
\end{align*}
where $C(n,s):=\int_{B_1}|h|^{2-n-2s}\,dh$.
The last integral in the right hand side is continuous as a function of $x$ (since $u$ is continuous).
Moreover, since the function inside the first integral is continuous in $x$ and non-negative (by the $C_0$-semiconvexity), the first integral
is lower semicontinuous as a function of $x$ by Fatou's lemma. This proves the claim.

Now, we remark that $-\Df u(t,x_0)=+\infty$ whenever $x_0$ is a point such that the subdifferential of $u(t)$ at $x_0$ is not single valued. Indeed, suppose that
$$
u(t,x) \geq \var_{x_0,p_1,p_2}(x):=\Big[u(t,x_0)+ \max\bigl\{p_1\cdot (x-x_0) ,p_2\cdot (x-x_0) \bigr\}-\frac{C_0}{2}|x-x_0|^2\Big]\chi_{B_{1}(x_0)}(x),
$$
for some $p_1\neq p_2$. Then, it is easy to check by a simple explicit computation that
$$
-\Df \var_{x_0,p_1,p_2}(x_0)=+\infty\qquad \forall\,s \geq 1/2.
$$
Hence,
since $u(t)\geq \var_{x_0,p_1,p_2}$ with equality at $x_0$, we get
$$
\Df u (t,x_0)\geq \Df \var_{x_0,p_1,p_2}(x_0)=+\infty.
$$
However, since $-\Df u(t)$ is bounded by $\|\Df u_0\|_{L^\infty(\R^n)}+\|\Df \psi\|_{L^\infty(\R^n)}$ (see \eqref{eq:Linfty bound Df u})
and it is a lower semicontinuous function, the above inequality
is impossible.
Thus the subdifferential of $u(t)$ at $x$ is a singleton
at every point, i.e., $u(t)$ is $C^1$.
Finally, the last part of the statement follows by a simple compactness argument.
\end{proof}

\begin{remark}\label{rmk:C1}{\rm
The spatial $C^1$-regularity proved in the above proposition is optimal for $s=1/2$.
Indeed, consider the case $n=1$ and $\psi\equiv 0$, and
use the interpretation of the $(1/2)$-fractional Laplacian
as the Dirichlet-to-Neumann operator for the harmonic extension, as explained
in Subsection \ref{subsect:prelim} (observe that $L_a=\Delta_{x,y}$ when $s=1/2$).
Then, we look for solutions to the problem
\begin{equation}
\label{eq:Df 12}
\left\{
\begin{array}{ll}
\min\{ u_t -u_y,u\}=0& \text{on }[0,T]\times\R,\\
\Delta_{x,y} u(t)=0&\text{on }[0,T]\times\R\times \R^+.\\
\end{array}
\right.
\end{equation}
Let us try to find traveling waves solutions to the above equation,
i.e., solutions of the form $u(t,x,y)=w(a t+x,y)$, with $a \in \R$.
In this case $u_t=a u_x$, so $w(x,y)$ has to solve
\begin{equation}
\label{eq:traveling}
\left\{
\begin{array}{ll}
a w_x(x,0)-w_y(x,0)=0 & \text{when } \{w(x,0)>0\},\\
\Delta_{x,y} w=0 &\text{on }\R\times \R^+.
\end{array}
\right.
\end{equation}
By using the complex variable $z=x+iy$ it is easy to construct $C^1$
solutions to the above equation:
if we denote $\rho=|z|=\sqrt{x^2+y^2}$ and $\theta=\arg(z)$, then
$w_\b(x,y):=-\rho^{1+\beta} \sin((1+\beta)\theta)=-{\rm Im}(z^{1+\beta})$ is harmonic in the half-space $y>0$
and solves
$$
\left\{
\begin{array}{ll}
w_\b(x,0)=0 & \text{on }\{x\geq 0\},\\
w_\b(x,0)>0,\,\frac{(w_\b)_x}{\tan(\b\pi)}=(w_\b)_y &\text{on }\{x><0\}.
\end{array}
\right.
$$
Observe that $w_\b$ is of class $C^{1+\b}$ both in space and time (but not more), and solves \eqref{eq:traveling} with
$a=1/\tan(\b\pi)$.
Since $\b \in (0,1)$ is arbitrary, we cannot expect to prove any uniform $C_{x}^{1+\a}$-regularity
for solutions to \eqref{eq:Df 12}. Thus, the $C^1_x$-regularity
proved in Proposition \ref{prop:C1x reg}
is optimal.

On the other hand, we observe that the case $u_t\geq0$
(i.e., the contact set shrinks in time) corresponds to $a \leq 0$, or equivalently to
$\beta \geq 1/2$. Hence, in this case the solutions constructed above 
are at least $C_{t,x}^{1+1/2}$, which is the optimal regularity result
for the stationary case \cite{athcaf,cafsalsil}. As we will show in the next section,
solutions to \eqref{eq:frac heat no initial}
satisfying $u_t\geq 0$ are of class $C^{1+1/2}$ in space.
In particular, by Lemma \ref{lemma:basic} this result applies to solutions of \eqref{eq:frac heat}.
}
\end{remark}

\section{Proof of Theorem \ref{thm:main}}
\label{sect:opt reg}

The strategy of the proof is the following:
first in Subsection \ref{sect:general C1a} we prove a general $C^{\a+2s}$-regularity result in space
which, roughly speaking, says the following: let $v:\R^n\to \R$ be a semiconvex function which touches from above
an obstacle $\psi:\R^n\to \R$ of class $C^2$. Assume that $\Df v$ is non-positive outside the contact set
and non-negative on the contact set. Then $v$ detaches from $\psi$ in a  $C^{\a+2s}$ fashion, for some $\alpha=\alpha(s)>0$
universal. In particular, as shown in Corollary \ref{cor:unif C1a v},
this implies that $\Df v\chi_{\{v=\psi\}} \in C_{x}^{\a}(\R^n)$.

Then, in Subsection \ref{subsect:monot} we use a monotonicity formula to prove
the optimal regularity in space
$$
\Df v\chi_{\{v=\psi\}} \in C_{x}^{1-s}(\R^n).
$$
Finally, in Subsection \ref{sect:C1s u} we apply the above estimate to any time slice $u(t)$
to prove that $\Df u(t)\chi_{u(t)=\psi}\in C_{x}^{1-s}(\R^n)$, uniformly in time. Then,
exploiting \eqref{eq:frac heat} and a bootstrap argument, we get \eqref{eq:opt reg}.

%Let us remark that $C^{1+s}$-regularity in space is the optimal regularity for the stationary case (see, for instance,
%\cite{silv,cafsalsil}). Moreover, when $s=1/2$, $C^{1+1/2}$-regularity in time is optimal too (see Remark \ref{rmk:C1}).
%Hence, our result is (almost) optimal in space-time when $s=1/2$.

\subsection{A general $C^{\a+2s}$-regularity result.}
\label{sect:general C1a}
In order to underline what are the key elements in the proof, in this and in the next subsection
we forget about equation \eqref{eq:frac heat}, and we work in the following general
setting:
let $v,\psi:\R^n \to \R$ be two globally Lipschitz functions with $v \geq \psi$. Assume that\footnote{
The smoothness assumption on $v$ inside the open set $\{v>\psi\}$ (see (A5))
is not essential for the proof of the regularity of $v$ at free boundary points, but it is only used to avoid some minor technical issues.
Anyhow this makes no differences for our purposes, since all the following results will be applied to $v=u(t)$ with $t>0$, and $u$ is smooth inside the open set $\{u>\psi\}$
(see the proof of Lemma \ref{lem:sign frac lapl}).}:
\begin{enumerate}
\item[(A1)] $\|D^2 \psi\|_{L^\infty(\R^n)}=:C_0 <+\infty$;
\item[(A2)] $\|\Df\psi\|_{C_x^{1-s}(\R^n)} <+\infty $;
\item[(A3)] $v-\psi,\Df v\in L^\infty(\R^n)$;
\item[(A4)] $v$ is $C_0$-semiconvex;
\item[(A5)] $v$ is smooth and $\Df v \leq 0$ inside the open set $\{v>\psi\}$;
\item[(A6)] $\|\Df\psi\|_{L^\infty(\R^n)} \geq \Df v \geq 0$ a.e. on $\{v=\psi\}$.
\end{enumerate}
Under these assumptions, we want to show that $v$ is $C^{\a+2s}$ at every free boundary point, with a uniform bound.
More precisely, we want to prove:
\begin{theorem}
\label{thm:unif C1a v} 
Let $v$ be as above. Then there exist $\bar C>0$ and $\a\in (0,1)$, depending on $C_0$, $\|\Df \psi\|_{C_x^{1-s}(\R^n)}$, $\|v-\psi\|_{L^\infty(\R^n)}$, and $\|\Df v\|_{L^\infty(\R^n)}$ only, such that
\begin{equation}
\label{eq:unif C1a v} 
\sup_{B_r(x)}|v-\psi| \leq \bar C\, r^{\a+2s},\quad \sup_{B_r(x)}|\Df v\chi_{\{v=\psi\}}| \leq \bar C\,
r^{\a}\qquad \forall \,r \leq 1
\end{equation}
for every $x \in \p \{u=\psi\}$.
\end{theorem}
Before proving the above result, let us show how it implies 
the following:
\begin{corollary}
\label{cor:unif C1a v}
Let $v$ be as above. Then there exist $\bar C'>0$ and $\a\in (0,1-s]$, depending on $C_0$, $\|v-\psi\|_{L^\infty(\R^n)}$, and $\|\Df v\|_{L^\infty(\R^n)}$ only, such that 
$$
\|\Df v \chi_{\{v=\psi\}}\|_{C_{x}^{\a}(\R^n)}\leq \bar C'.
$$
\end{corollary}
\begin{proof}
Without loss of generality, we can assume that the exponent $\a$ provided by Theorem \ref{thm:unif C1a v} is not greater than $1-s$.
Moreover, since $\Df v$ is bounded on $\{v=\psi\}$ (see (A6)), it suffices to control $|\Df v(x_1)-\Df v(x_2)|$ when $x_1,x_2 \in \{v=\psi\}$ and $|x_1-x_2|\leq 1/4.$

Let $M:=\|v-\psi\|_{L^\infty(\R^n)}$.
Moreover, given $x \in \{v=\psi\}$, let $d_F(x)$ denote its distance from the free boundary $\p\{v=\psi\}$.

Fix $x_1,x_2 \in \{v=\psi\}$, with $|x_1-x_2|\leq 1/4$. Two cases arise.\\
$\bullet$ \textit{Case 1:} $\max_{i=1,2}d_F(x_i) \geq 4|x_1-x_2|$.
Set $\tilde v:=v-\psi$. Since $\alpha\leq 1-s$ by assumption, thanks to (A2) it suffices to estimate $\Df \tilde v$ inside $\{\tilde v=0\}=\{v=\psi\}$.
Now, by Theorem \ref{thm:unif C1a v} we have
$$
\sup_{B_r(x_i)}|\tilde v| \leq \bar C\, r^{\a+2s},\qquad \forall\,r \leq 1, \,i=1,2.
$$
Hence, since $\tilde v=0$ inside $B_{4|x_1-x_2|}(x_1)\cap B_{4|x_1-x_2|}(x_2)$ and $|\tilde v|\leq M$ outside $B_1(x_1)\supset B_{1/2}(x_2)$,
we get
\begin{align*}
&|\Df \tilde v(x_1) - \Df \tilde v(x_2)|\\
&= \int_{\R^n} \frac{\tilde v(x') - \tilde v(x_1)}{|x'-x_1|^{n+2s}}\,dx'
- \int_{\R^n}\frac{\tilde v(x') - \tilde v(x_2)}{|x'-x_2|^{n+2s}}\,dx'\\
&\leq \int_{\R^n\setminus [B_{4|x_1-x_2|}(x_1)\cap B_{4|x_1-x_2|}(x_2)]}
|\tilde v(x')|\,\biggl|\frac{1}{|x'-x_1|^{n+2s}}-\frac{1}{|x'-x_1|^{n+2s}}\biggr|\,dx'\\
&\leq \bar C\,|x_1-x_2|\int_{B_1(x_1)\setminus [B_{4|x_1-x_2|}(x_1)\cap B_{4|x_1-x_2|}(x_2)]} |x'|^{\a+2s}
\biggl[\frac{1}{|x'-x_1|^{n+2s+1}}+\frac{1}{|x'-x_2|^{n+2s+1}}\biggr]\,dx'\\
&\qquad+ 2M\,|x_1-x_2|\int_{\R^n\setminus B_1(x_1)}
\biggl[\frac{1}{|x'-x_1|^{n+2s+1}}+\frac{1}{|x'-x_2|^{n+2s+1}}\biggr]\,dx'\\
&\leq C\,\biggl[\bar C\int_{|x_1-x_2|}^1 s^{\a-2}\,ds +M\biggr]\,|x_1-x_2|  \leq C\,|x_1-x_2|^\a,
\end{align*}
as desired.\\
$\bullet$ \textit{Case 2:} $\max_{i=1,2}d_F(x_i) \leq 4|x_1-x_2|$.
For every $i=1,2$, let $\bar x_i \in \p \{v=\psi\}$ denote a point such that $|x_i-\bar x_i|=d_F(x_i)$.
Then, by Theorem \ref{thm:unif C1a v} we get
\begin{align*}
|\Df v(x_1)-\Df v(x_2)| 
& \leq \sup_{B_{4|x_1-x_2|}(\bar x_1)}|\Df v|+\sup_{B_{4|x_1-x_2|}(\bar x_2)}|\Df v|\\
&\leq 8\bar C\,
|x_1-x_2|^{\a}.
\end{align*}
\end{proof}

\subsubsection{Proof of Theorem \ref{thm:unif C1a v}}
The strategy of the proof is analogous to the one used in \cite{athcaf} to study
the stationary fractional obstacle problem with $s=1/2$ (also called ``Signorini problem'').

With no loss of generality, we can assume that $0$ is a free boundary point,
and we prove \eqref{eq:unif C1a v} at $x=0$.
Moreover, by a slight abuse of notation, let us still denote by $v:\R^n\times \R^+ \to \R$ the $L_a$-harmonic extension of $v$,
i.e.,
$$
L_a v(x,y)=\div_{x,y}\bigl(y^a \nabla_{x,y} v(x,y)\bigr)=0\qquad \text{for $y>0$},
$$
and $v(x,0)=v(x)$, with $v(x)$ as above. Then
$$
\lim_{y\to 0^+}y^a v_y(x,y)=-\Df v(x,0),\qquad a=1-2s
$$
(see \eqref{eq:equivalent frac lapl}).
Let us observe that the $C_0$-semiconvexity of $v(x,0)$ (see (A4)) propagates in $y$: since
$$
v(x+h,0)+v(x-h,0)-2v(x,0)\geq -2C_0|h|^2 \qquad\forall\,h \in \R^n,
$$
the maximum principle implies
$$
v(x+h,y)+v(x-h,y)-2v(x,y)\geq -2C_0|h|^2\qquad\forall\,h \in \R^n,\,y>0,
$$
that is,
\begin{enumerate}
\item[(A7)] $v(\cdot,y)$ is $C_0$-semiconvex for all $y \geq 0$.
\end{enumerate}
In particular, since $L_a v=0$ we get
\begin{enumerate}
\item[(A8)] $\p_y(y^a v_y)\leq nC_0 y^a.$
\end{enumerate}
In the sequel, we will informally call the above property ``$a$-semiconcavity'' in $y$\footnote{
Even if we use the names ``$a$-semiconcavity'' and ``$C_0$-semiconvexity''
with different meanings, this should create no confusion.
Observe also that, when $a=0$, (A8) reduces to the classical notion of semiconcavity.}.
Set now
$$
\tilde v(x,y):=v(x,y)-\psi(x),
$$
and denote by $\L:=\{\tilde v(x,0)=0\}=\{v(x,0)=\psi(x)\}$ the contact set.
Observe that $ \tilde v_y=v_y$, which together with 
thanks to (A1)-(A6) gives that the function $\tilde v$ enjoys the following properties:
\begin{enumerate}
\item[(B1)] $\tilde v(x,0)\geq 0$
%\item[(B2)] $L_a \tilde v(x,y)=y^a(\Delta \psi(x)-\Delta\psi(0))=:y^ag(x)$
for all $(x,y)\in \R^n\times \R^+\setminus \L\times\{0\}$.
\item[(B2)] $\p_y(y^a\tilde v_y(x,y))\leq 2nC_0y^a$,  $\tilde v(\cdot,y)$ is $(2C_0)$-semiconvex for all $y \geq 0$.
\item[(B3)] $\lim_{y \to 0^+} y^a\tilde v_y(x,y) \leq 0$ for a.e. $x \in \L$,
$\lim_{y \to 0^+} y^a\tilde v_y(x,y)\geq  0$ for a.e. $x \in \R^n\setminus \L$.
\item[(B4)] $\tilde v(x,y)-\tilde v(x,0)\leq \frac{nC_0}{1+a}\,y^{2}$ for all $x \in \L$.
\item[(B5)] if $\tilde v(x,y)\geq h$, then $\tilde v(x,y) \geq h-C_0\rho^2$ in the half-ball
$$
HB_\rho(x):=\{z \in B_\rho(x)\subset \R^n\,:\,\<\n_x \tilde v(x,y),z-x\>\geq 0\}.
$$
\end{enumerate}
Observe that the proof of (B1)-(B5) is almost immediate, except for (B4) for which a (simple) computation is needed:
using (B2) and (B3), for a.e. $x \in \L$ we have
$$
\tilde v(x,y)-\tilde v(x,0)=\int_0^y \tilde v_y(x,s)\,ds =\int_0^y \frac{s^a\tilde v_y(x,s)}{s^a}\,ds \leq \int_0^y \frac{\int_0^s 2nC_0 \tau^a\,d\tau}{s^a}\,ds=\frac{nC_0}{1+a}\,y^2,
$$
and by continuity the above inequality holds for all $x \in \L$.

We use the notation $\G_r:=B_r\times [0,\eta_{n,a}r]$, where $\eta_{n,a}=\sqrt{\frac{1+a}{2n}}$.
We first show a decay result for $y^{a}\tilde v_y$:
\begin{proposition}
\label{prop: C1a v}
There exist two constants $K_1>0$, $\mu \in (0,1)$, depending on $C_0$, $\|v-\psi\|_{L^\infty(\R^n)}$, and $\|\Df v\|_{L^\infty(\R^n)}$ only, such that 
\begin{equation}
\label{eq:induction}
\inf_{\G_{4^{-k}}} y^a\tilde v_y \geq -K_1\,\mu^k,
\end{equation}
\end{proposition}
\begin{proof}
We prove the result by induction. \\
\textit{Case $k=1$:} since $v$ is $L_{a}$-harmonic, $y^a\tilde v_y=y^a v_y$ solves the ``conjugate'' equation $L_{-a}(y^a \tilde v_y)=0$ inside $\R^n\times \R^+$
(see for instance \cite[Subsection 2.3]{cafsil}). Hence, the boundedness of $\lim_{y\to 0^+}y^a\tilde v_y(x,y)=\Df v(x,0)$
(see (A6) and \eqref{eq:equivalent frac lapl}) combined with the maximum principle implies the result.\\
\textit{Induction step:} Assume the result is true for $k= k_0$, i.e.,
$$
\inf_{\G_{4^{-k_0}}} y^a\tilde v_y \geq -K_1\,\mu^{k_0}
$$
for some constants $K_1>0$ and $\mu\in(0,1)$ which will be chosen later,
%Then by Lemma \ref{lemma:deriv to fct}
%$$
%\sup_{\G_{4^{-(k_0^+2)}}}|\tilde v|\leq M\,\left(\frac{\mu}{4^{2s}}\right)^{k_0}.
%$$
and renormalize the solution inside $\G_1$ by setting 
$$
\tilde V(x,y):=\frac{1}{K_1} \left(\frac{4^{2s}}{\mu}\right)^{k_0} \tilde v\left(\frac{x}{4^{k_0}},\frac{y}{4^{k_0}}\right).
$$
It will also useful to consider the $L_a$-harmonic function
$$
\bar V(x,y):=\frac{1}{K_1} \left(\frac{4^{2s}}{\mu}\right)^{k_0} \bar v\left(\frac{x}{4^{k_0}},\frac{y}{4^{k_0}}\right),
$$
where $\bar v$ is the $L_a$-harmonic function given by
\begin{equation}
\label{eq:bar v}
\bar v(x,y):=v(x,y)-\psi(0)-\n \psi(0)\cdot x.
\end{equation}
Then, thanks to (A1) and (B2):
\begin{enumerate}
\item[(i)] $|\tilde V(x,y) - \bar V(x,y)|\leq \frac{C_0}{K_1(4^{2(1-s)}\mu)^{k_0}}|x|^2$ and $\tilde V_y=\bar V_y$;
\item[(ii)] $\inf_{\G_1} y^a \tilde V_y=\inf_{\G_1} y^a \bar V_y\geq -1$;
\item[(iii)] $\p_y(y^a\tilde V_{y})=\p_y(y^a\bar V_{y})\leq \frac{2nC_0}{K_1(4^{2(1-s)}\mu)^{k_0}}y^a$,
$\tilde V$ and $\bar V$ are $\left(\frac{2C_0}{K_1(4^{2(1-s)}\mu)^{k_0}}\right)$-semiconvex inside $\G_1$.
\end{enumerate}
Fix $L:=\bar C_{n,a} C_0$, where $\bar C_{n,a}\gg 1$ is a large constant depending on $n$ and $a$ only
(to be fixed later), and define
$$
\bar W(x,y):=\bar V(x,y) - \frac{L}{K_1(4^{2(1-s)}\mu)^{k_0}}\biggl[|x|^2 - \frac{n}{1+a}\,y^2 \biggr].
$$
Thanks to (B1)-(B3), the function $\bar W$ satisfies the following properties:
\begin{enumerate}
\item[1.] it is $L_a$-harmonic in the interior of $\G_{1/8}$;
\item[2.] $\bar W(x,0)< 0$ for $x\in \bigl(\L\setminus \{0\}\bigr)\times\{0\}$;
\item[3.] $\lim_{(x,y)\to (0,0)} \bar W(x,y)=0$;
\item[4.] $\lim_{y\to 0^+}y^a\bar W_y(x,y)\geq 0$ for $x \not\in B_{1/8}\setminus \L$.
\end{enumerate}
Hence, up to replacing $\bar W$ by $\bar W+\e y^{1-a}$ with $\e>0$
(so that the inequality in 4. becomes strict) and then letting $\e\to 0$, by Hopf's Lemma $\bar W$ attains its non-negative maximum on
$\p \G_{1/8} \setminus \{y=0\}$.\\
Two cases arise:\\
\textit{Case 1:} \textit{The maximum is attained on $\p \G_{1/8}\cap \{y=\eta_{n,a}/8\}$}.\\
In this case, there exists $x_0 \in B_{1/8}$ such that
$$
\bar V\left(x_0,\textstyle{\frac{\eta_{n,a}}{8}}\right) \geq -c'_{n,a}\frac{L}{K_1(4^{2(1-s)}\mu)^{k_0}},
$$
for some constant $c_{n,a}'>0$ depending on $n,a$ only.
Thanks to the semiconvexity in $x$ (see property (iii) above) and recalling that $L \gg C_0$ by assumption,
there exists an $n$-dimensional half-ball $HB_{1/2}\left(x_0,\textstyle{\frac{\eta_{n,a}}{8}}\right)$ 
such that
$$
\bar V\left(x,\textstyle{\frac{\eta_{n,a}}{8}}\right) \geq -\frac{c_{n,a}'}2\frac{L}{K_1(4^{2(1-s)}\mu)^{k_0}}\qquad
\forall\, x \in HB_{1/2}\left(x_0,\textstyle{\frac{\eta_{n,a}}{8}}\right)
$$
(see property (B5)).
Recall now that $\lim_{y\to 0^+}y^a\bar V_y(x,y)=\lim_{y\to 0^+}y^a\tilde V_y(x,y)\geq 0$ when $\tilde V(x,0)>0$,
while $\lim_{y\to 0^+}y^a\bar V_y(x,y)=\lim_{y\to 0^+}y^a\tilde V_y(x,y)\leq 0$ when $\tilde V(x,0)=0$.
Hence, by the ``$a$-semiconcavity'' of $\bar V$ in $y$ (property (iii) above) and by (i), it is easy to see that 
$$
\lim_{y\to 0^+}y^a\bar V_y\left(x,y\right) \geq -C_{n,a}''\frac{L}{K_1(4^{2(1-s)}\mu)^{k_0}}\qquad
\forall\, x \in HB_{1/2}\left(x_0,0\right),
$$
for some universal constant $C_{n,a}''>0$.\\
\textit{Case 2:} \textit{The maximum is attained on $\p \G_{1/8}\setminus \{y=\eta_{n,a}/8\}$}.\\
Let $(x_0',y_0')$ be a maximum point. Since such a point belongs to the lateral side of the
cylinder, recalling the definition of $\eta_{n,a}$ we have
$|x_0'|^2\ge \frac{2n}{1+a}|y_0'|^2$,
which implies
$$
\bar V(x_0',y_0')\geq \frac{L}{K_1(4^{2(1-s)}\mu)^{k_0}}.
$$
Again, we recall that
$\lim_{y\to 0^+}y^a\bar V_y(x,y)=\lim_{y\to 0^+}y^a\tilde V_y(x,y) \geq 0$ when $\tilde V(x,0)>0$,
while $\lim_{y\to 0^+}y^a\bar V_y(x,y)=\lim_{y\to 0^+}y^a\tilde V_y(x,y)\leq 0$ when $\tilde V(x,0)=0$. Thus,
by the half-ball estimate (B5) applied to $\bar V$, by the ``$a$-semiconcavity'' of $\bar V$ in $y$ (property (iii)) and by (i), we obtain
$$
\lim_{y\to 0^+}y^a\bar V_y\left(x,y\right) \geq 0 \qquad \forall\, x \in HB_{1/2}\left(x_0',0\right).
$$

Hence, in both case we have reached the following conclusion:\\
\textit{There exist a constant $C_1>0$, depending on $n,a$, and $C_0$ only, and a point $\bar x \in B_{1/8}\subset \R^n$, such that
$$
\lim_{y\to 0^+}y^a\bar V_y(x,0) >-\frac{C_1}{K_1(4^{2(1-s)}\mu)^{k_0}} \qquad \forall\,x \in HB_{1/2}\left(\bar x,0\right).
$$}
Thus, if we choose $K_1$ and $\mu$ satisfying $K_1>2C_1$ and $\mu \geq 1/4^{2(1-s)}$, then we obtain
\begin{equation}
\label{eq:wy 12}
\lim_{y\to 0^+} y^a\bar V_y(x,y) >-\frac{1}{2}.
\end{equation}
Moreover, thanks to (ii),
\begin{equation}
\label{eq:wy 1eps}
y^a \bar V_y \geq -1\qquad \text{in }\Gamma_1.
\end{equation}
As we already observed before, the fact that $\bar V$ is $L_{a}$-harmonic implies that
$y^a\bar V_y$ solves the conjugate equation $L_{-a}(y^a \bar V_y)=0$ inside $\R^n\times \R^+$. Hence, thanks to \eqref{eq:wy 12} and \eqref{eq:wy 1eps}, the Poisson representation
formula (see \cite[Subsection 2.4]{cafsil}) implies the existence of a constant $\theta<1$ such that
$$
y^a\bar V_y\left(x,\textstyle{\frac{\eta_{n,a}}{4}}\right) \geq -\theta \qquad \forall\, x\in B_{1/4}.
$$
Therefore, by (i) and the ``$a$-semiconcavity'' of $\bar v$ in $y$ (property (iii)), we obtain
$$
y^a\tilde V_y(x,y)=y^a\bar V_y(x,y) \geq -\theta - \frac{2nC_0}{K_1(4^{2(1-s)}\mu)^{k_0}}=:-\mu>-1,
$$
provided $K_1$ is sufficiently large.
Rescaling back, this proves \eqref{eq:induction} with $k=k_0+1$, which concludes the proof.
\end{proof}

Recalling that $y^a \tilde v_y=y^a v_y$,
thanks to \eqref{eq:equivalent frac lapl}
 and (A6) the above proposition implies
$$ 
\sup_{B_r(x)}|\Df v\chi_{\{v=\psi\}}| \leq \bar C\,
r^{\a}\qquad \forall \,r \leq 1.
$$
We now show that a control from below on $y^a\tilde v_y$
inside $\G_r$ gives
a control from both sides on $\tilde v$ inside $\G_{r/8}$.
This
will conclude the proof of Theorem \ref{thm:unif C1a v}.

\begin{lemma}
\label{lemma:deriv to fct}
Fix $K>0$, $\a\in (0,1)$, and assume that 
\begin{equation}
\label{eq:bound below vy}
\inf_{\G_r}y^a\tilde v_y \geq -Kr^\alpha
\end{equation}
for some $r \in (0,1]$.
Then there exists a constant $M=M(K,\alpha,C_0)$, independent of $r$, such that
$$
\sup_{\G_{r/8}}|\tilde v| \leq M r^{1+\alpha-a}=M r^{\alpha+2s}.
$$
\end{lemma}
\begin{proof}
Since $\tilde v$ is globally bounded, it suffices to prove the result for $r$ small.
First of all, let us observe that, thanks to (B1) and \eqref{eq:bound below vy},
\begin{equation}
\label{eq:bound below v}
\tilde v(x,y) \geq \tilde v(x,0) - K\int_0^y \frac{r^\a}{u^{a}}\,du \geq  -K\,\frac{r^{1+\a-a}}{1-a}\qquad
\forall \,(x,y)\in \G_r,
\end{equation}
which proves the lower bound on $\tilde v$. 

To prove the upper-bound, assume that there exists a point $(\bar x,\bar y)\in \G_{r/8}$ such that
$\tilde v(\bar x,\bar y)\geq Mr^{1+\a-a}$ for some large constant $M$.
Arguing as above, this implies
\begin{equation}
\label{eq:bound below big v}
\tilde v\left(\bar x,\frac{\eta_{n,a}r}{2}\right)\geq \frac{M}4 \,r^{1+\a-a},
\end{equation}
provided $M$ is sufficiently large (depending only on $K$).
Now, let $B':=B_{\frac{\eta_{n,a}r}{2}}'\left(\bar x,\frac{\eta_{n,a}r}{2}\right)$ denote the $(n+1)$-dimensional ball
of radius $\frac{\eta_{n,a}r}{2}$
centered at $\left(\bar x,\frac{\eta_{n,a}r}{2}\right)\in \R^n\times\R^+$, and set
$B'/2:=B_{\eta_{n,a}r/4}'\left(\bar x,\frac{\eta_{n,a}r}{2}\right)$.
Then $B'\subset \G_r$ and $\left(0,\frac{\eta_{n,a}r}{2}\right)\in B'/2$.

Let $\bar v$ be as in \eqref{eq:bar v}.
Thanks to (A1), $|\bar v-\tilde v|\leq C_0\,r^{2}$ inside $\G_r$ (observe that $\eta_{n,a} \leq 1$),
which together with \eqref{eq:bound below v}
implies that $w+K\frac{r^{1+\a-a}}{1-a}+C_0r^{2}$ is non-negative inside $\G_r$.
Hence we can apply Harnack inequality inside $B'$ (see \cite[Proposition 2.2]{cafsalsil} and \cite{fabkenser}) to obtain
$$
\frac{M}8 \,r^{1+\a-a} \leq \sup_{B'/2} \biggl[\bar v+K\frac{r^{1+\a-a}}{1-a}+C_0r^{2}\biggr] \leq C \biggl[\bar v\left(0,\frac{\eta_{n,a}r}{2}\right)+K\frac{r^{1+\a-a}}{1-a}+C_0r^{2}\biggr],
$$
that is $\bar v(0,\eta_{n,a}r/2)\geq c_0Mr^{1+\a-a}-K\frac{r^{1+\a-a}}{1-a}-C_0r^{2}$ for some universal constant $c_0>0$,
which gives
$$
\tilde v\left(0,\frac{\eta_{n,a}r}{2}\right)\geq c_0M\,r^{1+\a-a} -K\frac{r^{1+\a-a}}{1-a}-2C_0r^{2}.
$$
Since $0\in \L$, combining the above estimate with property (B4) we get
$$
0=\tilde v(0,0)\geq \tilde v\left(0,\frac{\eta_{n,a}r}{2}\right) - \frac{nC_0}{1+a}r^{2}\geq
c_0M\,r^{1+\a-a}-K\frac{r^{1+\a-a}}{1-a} - 
\left[\frac{nC_0}{1+a}+2C_0\right]r^{2},
$$
which shows that $M$ is universally bounded, as desired.
\end{proof}

\subsection{Towards optimal regularity: a monotonicity formula}
\label{subsect:monot}
We use the same notation as in the previous subsection.

We have proved that $\Df v\chi_{\{v=\psi\}}$ grows at most as $r^\a$ near any free boundary point,
which implies that  $\Df v \chi_{\{v=\psi\}}\in C_{x}^{\a}(\R^n)$ (see Corollary \ref{cor:unif C1a v}).
Consider now the function $w:\R^n\times \R^+\to \R$ obtained by solving the Dirichlet problem
\begin{equation}
\label{eq:def w}
\left\{
\begin{array}{ll}
L_{-a}w=0 &\text{on }\R^n\times \R^+,\\
w(x,0)=\Df v(x) \chi_{\{v=\psi\}}(x)&\text{on }\R^n.
\end{array}
\right.
\end{equation}
Since $w(x,0)\geq 0$, the maximum principle implies $w \geq 0$ everywhere.

Assume that $0\in \R^n$ is a free boundary point.
Since $\Df v(x)$ is globally bounded (see (A3)), using the Poisson representation formula for $w$ \cite[Subsection 2.4]{cafsil} together with
the uniform $C_{x}^{\a}$-regularity of $w(x,0)$ (Corollary \ref{cor:unif C1a v}) we get
$$
\sup_{|x|^2+y^2\leq r^2} w(x,y) \leq C \,r^{\a},
$$
for some uniform constant $C$.
The goal of this subsection is to show that
\begin{equation}
\label{eq:C1s}
\sup_{|x|^2+y^2\leq r^2} w(x,y) \leq \tilde C \,r^{1-s},
\end{equation}
for some constant $\tilde C>0$, depending on $C_0$, $\|\Df \psi\|_{C_x^{1-s}(\R^n)}$, $\|v-\psi\|_{L^\infty(\R^n)}$, and $\|\Df v\|_{L^\infty(\R^n)}$ only.

This estimate will imply that $\Df v$  grows at most as $|x|^{1-s}$ at every free boundary point, so that
the same proof as in Corollary \ref{cor:unif C1a v} will give that $\Df v \chi_{\{v=\psi\}}\in C_{x}^{1-s}(\R^n)$, with a uniform bound.
Then, in the next subsection we will apply this estimate to $v=u(t)$ for every $t \in (0,T]$,
and using \eqref{eq:frac heat} we will obtain the desired regularity result for $u$.\\

As in the previous subsection, we consider the function $\tilde v(x,y)=v(x,y) - \psi(x)$.
Thanks to Theorem \ref{thm:unif C1a v} together with the $(2C_0)$-semiconvexity of $\tilde v$ (see (B2) in the previous subsection),
we can mimic the proof of \cite[Lemma 5]{athcaf}:
\begin{lemma}
\label{lem:apply monot}
Let $\bar C>0$ and $\a \in (0,1-s]$ be as in Theorem \ref{thm:unif C1a v}, and set $\delta_\a=\delta_\a(s):=\frac{1}{4}\left(\frac{\a}{\a+2s} -\frac{\a}{2}\right)$.
Then there exists $r_0=r_0(\a,s,\bar C,C_0)>0$
such that the convex hull of the set $\{x \in \R^n\,:\,w(x,0)\geq r^{\a+\delta_\a}\}$ in $B_r\subset  \R^n$
does not contain the origin
for $r \leq r_0$.
\end{lemma}
\begin{proof}
Thanks to (B3),
$$
\tilde v(x,0)=0\quad \text{and}\quad\lim_{y\to 0^+}y^a\tilde v_y(t,x,y)=-w(x,0) \leq -r^{\a+\delta_\a}
\qquad \forall\,x \in \{w(x,0)\geq r^{\a+\delta_\a}\}.
$$
Hence, by the ``$a$-semiconcavity'' (B2) of $\tilde v$ in $y$, for any $x \in \{w(x,0)\geq r^{\a+\delta_\a}\}$ we have
\begin{equation}
\label{eq:upper bound tilde v}
\begin{split}
\tilde v(x,h) &\leq \int_0^h \frac{s^a\tilde v_y(t,x,s)}{s^a}\,ds\\
&\leq -\int_0^h \frac{r^{\a+\delta_\a}}{s^a}\,ds+ \int_0^h \frac{1}{s^a}\biggl(\int_0^s 2nC_0 \tau^a\,d\tau \biggr)\,ds\\
&=-\frac{1}{1-a}r^{\a+\d_\a} h^{1-a} + \frac{nC_0}{1+a}h^2=-\frac{1}{2s}r^{\a+\d_\a}h^{2s}  + \frac{nC_0}{1+a}h^2.
\end{split}
\end{equation}
On the other hand, Theorem \ref{thm:unif C1a v} gives
\begin{equation}
\label{eq:lower bound tilde v}
\tilde v(0,h)=\tilde v(0,h)-\tilde v(0,0) \geq -\bar Ch^{\a+2s}.
\end{equation}
Assume now by contradiction that the convex hull of the set $\{(x,0)\,:\,w(x,0)\geq r^{\a+\delta_\a}\}\cap B_r$
contains $(0,0)$. Then, by the $(2C_0)$-semiconvexity of $\tilde v(\cdot,h)$ (see (B2)) we get
$$
\tilde v(0,h) \leq \sup_{x \in \{w(x,0)\geq r^{\a+\delta_\a}\}} \tilde v(x,h) +C_0r^2,
$$ 
which together with \eqref{eq:lower bound tilde v} and \eqref{eq:upper bound tilde v} gives
$$
\bar C h^{\a+2s} \geq \frac{1}{2s}r^{\a+\d_\a}h^{2s} - \frac{nC_0}{1+a}h^2- C_0r^2
$$
for all $r,h\in (0,1)$.
To get a contradiction from the above inequality, we want to choose $h=h(r)$ in such a way that
$$
h^2 \ll r^2 \ll h^{\a+2s} \ll r^{\a+\d_\a}h^{2s}\qquad \text{for $r$ sufficiently small}.
$$
To this aim, set $h=r^{1+2\d_\a/\a}$. Then $h^\a=r^{\a+2\d_\a}=o(r^{\a+\d_\a})$, and both the first and the third condition above hold.
To ensure that also the second one is satisfied, it suffices to have
$$
(\a+2s)\biggl(1+2\frac{\d_\a}{\a}\biggr) <2,
$$
that is
$$
\d_\a< \frac{1}{2}\left(\frac{\a}{\a+2s} -\frac{\a}{2}\right).
$$
Recalling that  $\a +2s <2$ (so, the right hand side is positive) and the definition of $\delta_{\a}$, we get the desired
contradiction, which concludes the proof.
\end{proof}

We now want to use a monotonicity formula to improve the decay of
$w(x,y)$ at the origin.
We first need some preliminary results:
\begin{lemma}
\label{lemma:prelim for monot}
(i) There exists a constant $C'$, depending on $C_0$, $\|\Df \psi\|_{C_x^{1-s}(\R^n)}$, $\|v-\psi\|_{L^\infty(\R^n)}$, and $\|\Df v\|_{L^\infty(\R^n)}$ only, such that  
$$
\limsup_{y \to 0^+} \int_{B_r} (w^2)_y(x,y) \frac{y^{-a}}{(|x|^2+y^2)^{(n-1+a)/2}}\,dx
\geq -C'r^{\a+1+a}\qquad \forall\,r \geq 0.
$$
(ii) It holds
$$
\lim_{y \to 0^+}
\int_{B_r} w^2(x,y) \p_y \left( \frac{1}{(|x|^2+y^2)^{(n-1-a)/2}}\right)\,y^{-a}\,dx=0.
$$
\end{lemma}
\begin{proof}
\textit{(i)} To show the estimate, let us observe that:
\begin{enumerate}
\item[(1)] Since $w(\cdot,0)=0$  on $\R^n\setminus \L$ while $w(\cdot,0)\geq 0$ on $\L$ (see (A6)), by the maximum principle we get
$w(x,y)\geq 0$. Hence
$$
w(x,y)\geq w(x,0) \qquad \forall \,x \in \R^n\setminus \L,\,y>0.
$$
\item[(2)] By the $a$-semiconcavity of $v$ in $y$ (see (A8)), 
$$
y^a v_y(x,y) \leq  \lim_{s\to 0^+}s^a v_y(x,s) +\frac{nC_0}{1+a} y^{1+a}.
$$
(Observe that the above limit always exists, since $\Df v(x,0)$ is H\"older continuous 
on the contact set, while $v$ is smooth outside, see (A5).)
\item[(3)] 
The function $y^a v_y$ solves
$$
\left\{
\begin{array}{l}
L_{-a}(y^a v_y)=0,\\
\lim_{y\to 0^+}y^a v_y(x,0)=-\Df v(x,0),
\end{array}
\right.
$$
(see \cite[Subsection 2.3]{cafsil}).
Since $w(x,0)\geq \Df v(x,0)$ by (A5), the maximum principle gives $w\geq -y^a v_y$ on $\R^n\times \R^+$.
Hence, since $w(x,0)=-\lim_{y\to 0^+}y^a v_y(x,0)$ in $\L$, by (2) above we get
$$
w(x,y) \geq  w(x,0) - \frac{nC_0}{1+a}y^{1+a}\qquad \forall\, x\in \L,\,y>0.
$$
\end{enumerate}
Combining (1) and (3) we obtain
\begin{equation}
\label{eq:bound below w}
w(x,y) \geq  w(x,0) - \frac{nC_0}{1+a}y^{1+a}\qquad \forall\, x\in \R^n,\,y>0.
\end{equation}
This estimate, together with the $C_{x}^{\a}$ regularity of $w$ and the fact that $w$ is non-negative, implies that, for all $x \in B_r$ and $y>0$,
\begin{equation}
\label{eq:bound 1+a}
w^2(x,y)-  w^2(x,0)=[w(x,y)-  w(x,0)][w(x,y)+  w(x,0)]\geq -K y^{1+a} (r+y)^\a,
\end{equation}
for some uniform constant $K>0$.

We now want to estimate from below
$$
\limsup_{y \to 0^+} \int_{B_r} (w^2)_y(x,y) \frac{y^{-a}}{(|x|^2+y^2)^{(n-1-a)/2}}\,dx.
$$
To this aim, consider the change of variable $s=s(y):=\left(\frac{y}{1+a}\right)^{1+a}$ and define
$\tilde w(x,s(y)):=w(x,y)$. Then \eqref{eq:bound 1+a} becomes
\begin{equation}
\label{eq:bound 1}
\tilde w^2(x,s)-  \tilde w^2(x,0)\geq -K' s (r+s^{1/(1+a)})^\a \qquad \forall \,x \in B_r,\,s>0,
\end{equation}
for some uniform constant $K'>0$.
Moreover, since $y^{-a}(w^2)_y(x,y)=(\tilde w^2)_s(x,s)$, we are left with estimating
$$
\limsup_{s \to 0^+} \int_{B_r} (\tilde w^2)_s(x,s) \frac{1}{(|x|^2+(1+a)^2 s^{2/(1+a)})^{(n-1-a)/2}}\,dx.
$$
To do this, we average the above expression with respect to $s \in [0,\e]$ and we use Fubini Theorem to get
\begin{align*}
\frac{1}{\e}\int_0^\e ds &\int_{B_r} (\tilde w^2)_s(x,s) \frac{1}{(|x|^2+(1+a)^2 s^{2/(1+a)})^{(n-1-a)/2}}\,dx\\
&=\int_{B_r} \frac{1}{\e}\left[\frac{\tilde w^2(x,\e)}{(|x|^2+(1+a)^2 \e^{2/(1+a)})^{(n-1-a)/2}}-  \frac{\tilde w^2(x,0)}{|x|^{n-1-a}}\right] \,dx\\
&\qquad-\frac{1}{\e}\int_0^\e ds \int_{B_r}\tilde w^2(x,s) \frac{d}{ds}\left( \frac{1}{(|x|^2+(1+a)^2 s^{2/(1+a)})^{(n-1-a)/2}}\right)\,dx.
\end{align*}
Now, thanks to \eqref{eq:bound 1}, the $C_{x}^{\a}$-regularity of $w(x,0)=\tilde w(x,0)$,
and the fact that $$
\frac{d}{ds}\left( \frac{1}{(|x|^2+(1+a)^2 s^{2/(1+a)})^{(n-1-a)/2}}\right)\leq 0,
$$
we obtain that the above expression is bounded from below by
\begin{align*}
 \int_{B_r} &\frac{1}{\e}\left[\frac{\tilde w^2(x,0) - K'\e (r+\e^{1/(1+a)})^\a}{(|x|^2+(1+a)^2 \e^{2/(1+a)})^{(n-1-a)/2}}-  \frac{\tilde w^2(x,0)}{|x|^{n-1-a}}\right] \,dx\\
&\geq -K'(r+\e^{1/(1+a)})^\a\int_{B_r}\frac{1}{|x|^{n-1-a}} \,dx\\
&\quad+ C\int_{B_r}\frac{|x|^{2\a}}{\e}\left[\frac{1}{(|x|^2+(1+a)^2 \e^{2/(1+a)})^{(n-1-a)/2}}-  \frac{1}{|x|^{n-1-a}}\right] \,dx
\end{align*}
%We now distinguish two cases:\\
%\textit{Case 1: $a \geq 0$.} In this case we remark that
%$$
%\left|\frac{1}{(|x|^2+(1+a)^2 \e^{2/(1+a)})^{(n-1-a)/2}}-  \frac{1}{|x|^{n-1-a}}\right| \leq C \frac{\e^{2/(1+a)}}{|x|^{n-a}},
%$$
%and since the function $\frac{|x|^{2\a}}{|x|^{n-a}}$ is integrable near the origin we get
%\begin{align*}
%&\left|\int_{B_r}\frac{|x|^{2\a}}{\e}\left[\frac{1}{(|x|^2+(1+a)^2 \e^{2/(1+a)})^{(n-1-a)/2}}-  \frac{1}{|x|^{n-1-a}}\right] \,dx\right|\\
%&\leq C \e^{2/(1+a)-1}\int_{B_r} \frac{|x|^{2\a}}{|x|^{n-a}}\,dx\\
%&=C \e^{2/(1+a)-1} \to 0  \qquad
%\text{as $\e \to 0.$}
%\end{align*}
%\textit{Case 2: $a < 0$.} 
Concerning the first term in the right hand side,  since $a=1-2s<1$ we have
$$
(r+\e^{1/(1+a)})^\a\int_{B_r}\frac{1}{|x|^{n-1-a}} \,dx \to C_{n,a}r^{\a+1+a}=C_{n,a}r^{\a+1+a}\qquad \text{as }\e \to 0.
$$
For the second term, we want to prove that it converges to $0$ as $\e\to 0.$ To this aim, we split the integral into
two terms: the integral over $B_{\e^\beta}$, and the one over $B_r \setminus B_{\e^\beta}$, where $\beta>0$ has to be chosen.
For the first term, we can bound it from below by
$$
-\frac{C}{\e} \int_{B_{\e^\beta}} \frac{|x|^{2\a}}{|x|^{n-1-a}}\,dx = C \e^{\b (2\a + a+1) -1}.
$$
Thus, by choosing $\beta \in \left(\frac1{2\a + a+1},\frac1{1+a}\right)$ we ensure that the above expression converges to $0$ as $\e\to 0$.
Moreover, the fact that $\beta<1/(1+a)$ implies that
$$
\e^{2/(1+a)} \ll |x|^2 \qquad \forall\, |x|\geq \e^\b.
$$
Therefore, for estimating the second part we can use polar coordinates and the fact that
$$
(|x|^2+(1+a)^2 \e^{2/(1+a)})^{(n-1-a)/2} \sim |x|^{n-1-a}+C\e^{2/(1+a)}|x|^{n-3-a} \qquad \forall\, |x|\geq \e^\b
$$
to write
\begin{align*}
\int_{B_r\setminus B_{\e^\beta}}&\frac{|x|^{2\a}}{\e}\left[\frac{1}{(|x|^2+(1+a)^2 \e^{2/(1+a)})^{(n-1-a)/2}}-  \frac{1}{|x|^{n-1-a}}\right] \,dx\\
&\sim \frac{C}{\e} \int_{\e^\b}^r \rho^{n-1+2\a} \left[\frac{1}{\rho^{n-1-a}+C\e^{2/(1+a)}\rho^{n-3-a}}-  \frac{1}{\rho^{n-1-a}}\right]\,d\rho\\
&=\frac{C}{\e} \int_{\e^\b}^r \rho^{2\a+a} \left[\frac{\rho^2}{\rho^2+C\e^{2/(1+a)}}- 1\right]\,d\rho
=-\frac{C}{\e} \int_{\e^\b}^r \rho^{2\a+a} \frac{\e^{2/(1+a)}}{\rho^2+C\e^{2/(1+a)}}\,d\rho\\
&\geq -\frac{C\e^{2/(1+a)}}{\e} \int_{\e^\b}^r \rho^{2\a+a-2}\,d\rho\geq -C \e^{2/(1+a) - 1}\bigl[1+\e^{\beta(2\a+a-1)}\bigr]. \\
\end{align*}
Let us remark that $2/(1+a)>1$, so if
$2\a+a-1\geq 0$ the above expression obviously converges to $0$.
On the other hand, if  $2\a+a-1< 0$, 
since
$\b<1/(1+a)$ we get
$$
\frac2{1+a}- 1+\beta(2\a+a-1)> \frac2{1+a}- 1+\frac{2\a+a-1}{1+a} \geq \frac{2-1-a+2\a+a-1}{1+a}=\frac{2\a}{1+a}>0,
$$
and again the above expression converges to $0$.
All in all, we conclude that
$$
\int_{B_r}\frac{|x|^{2\a}}{\e}\left[\frac{1}{(|x|^2+(1+a)^2 \e^{2/(1+a)})^{(n-1-a)/2}}-  \frac{1}{|x|^{n-1-a}}\right] \,dx\to 0 \qquad \text{as $\e\to 0$},
$$
%We now observe that the trivial inequality $\sqrt{1+\ell^2} \leq 1+\ell$ for all $\ell>0$ gives
%$$
%\frac{1}{(|x|^2+(1+a)^2 \e^{2/(1+a)})^{(n-1+a)/2}}\geq \frac{1}{(|x|+(1+a)^2 \e^{1/(1+a)})^{n-1+a}}.
%$$
%Moreover, it is easy to compute the difference of the two integrals
%$$
%\int_{B_r}|x|^{2\a}\left[\frac{1}{(|x|+(1+a)^2 \e^{1/(1+a)})^{n-1-a}}-  \frac{1}{|x|^{n-1-a}}\right] \,dx
%$$
%in polar coordinates and to show that the above expression is bounded from below by
%$$
%-C \left[\bigl(\e^{1/(1+a)}\bigr)^{2\a + 1 +a}+  \e^{1/(1+a)}\right]
%$$
%when $\e \to 0$.
%Hence
%\begin{align*}
%&\int_{B_r}\frac{|x|^{2\a}}{\e}\left[\frac{1}{(|x|^2+(1+a)^2 \e^{2/(1+a)})^{(n-1+a)/2}}-  \frac{1}{|x|^{n-1+a}}\right] \,dx\\
%&\geq  -C \frac{\bigl(\e^{1/(1+a)}\bigr)^{2\a + 1 +a}+  \e^{1/(1+a)}}{\e} =\e^{2\a/(1+a)}\to 0 \qquad
%\text{as $\e \to 0.$}
%\end{align*}
%Furthermore
%$$
%(r+\e^{1/(1+a)})^\a\int_{B_r}\frac{1}{|x|^{n-1-a}} \,dx \to C_{n,a}r^{\a+1+a}=C_{n,a}r^{\a+1+a}.
%$$
%All in all we have proved that
so that combining all our estimates together we obtain
$$
\liminf_{\e\to 0}\frac{1}{\e}\int_0^\e ds \int_{B_r} (\tilde w^2)_s(x,s) \frac{1}{(|x|^2+(1+a)^2 s^{2/(1+a)})^{(n-1-a)/2}}\,dx \geq -K'C_{n,a}r^{\a+1+a}.
$$
From this fact we easily deduce that
$$
\limsup_{\e\to 0} \int_{B_r} (\tilde w^2)_s(x,\e) \frac{1}{(|x|^2+(1+a)^2 \e^{2/(1+a)})^{(n-1-a)/2}}\,dx \geq -K'C_{n,a}r^{\a+1+a},
$$
which concludes the proof of (i).\\
\textit{(ii)} In this case, we use the $C_{x}^{\a}$-regularity of $w$ to control the integral by
$$
y^{1-a}\int_{B_r} \frac{1}{(|x|^2+y^2)^{(n+1-a)/2-\a}} \,dx.
$$
Using polar coordinates, the above integral is comparable to
$$
y^{1-a}\int_0^r \frac{\rho^{n-1}}{(\rho^2+y^2)^{(n+1-a)/2-\a}}\,d\rho\sim y^{1-a}\int_0^r \frac{\rho^{n-1}}{(\rho+y)^{n+1-a-2\a}}\,d\rho \sim \frac{y^{1-a}}{y^{1-a-2\a}}=y^{2\a},
$$
and the above expression converges to $0$ as $y\to 0$.
\end{proof}

We will also need a result on the first eigenvalue of a weighted Laplacian on the half-sphere.
We use $\n_\theta w$ to denote the derivative of $w$ with respect to the angular variables:.
\begin{lemma}
\label{lem:eigen}
Set
$\S^{n+1}_+:=\S^{n+1}\cap \{x_{n+1}\geq 0\}$, $\S^{n+1}_0:=\p\S^{n+1}_+
=\S^{n+1}\cap \{x_{n+1}=0\}$, $\S^{n+1}_{0,+}:=\S^{n+1}\cap \{x_{n+1}=0\}\cap \{x_n \geq 0\}$. Then
$$
\inf_{h \in H^{1/2}(\S^{n+1}_0),\,h=0 \text{ on } \S^{n+1}_{0,+} }
\frac{\int_{\S^{n+1}_+}|\n_\theta h|^2y^{-a}\,d\sigma}{\int_{\S^{n+1}_+} h^2y^{-a}\,d\sigma}
=(1-s)(n-1+s).
$$
\end{lemma}
\begin{proof}
Let $\bar h(\theta)$ denote the restriction to $\S^{n+1}_+$ of
$$
\bar H(x,y):=\bigl(\sqrt{x_n^2+y^2}-x_n\bigr)^{1-s},
$$ 
that is $\bar H=r^{1-s}\bar h(\theta)$. 
As shown in \cite[Proposition 5.4]{cafsalsil}, $\bar h$ is the first eigenfunction corresponding to the above minimization problem.
Moreover, $\bar H$ solves $L_{-a} \bar H=0$ for $y>0$.\footnote{A simple way to check this fact is the following:
the function $G:=\bigl(\sqrt{x_n^2+y^2}-x_n\bigr)^{1/2}$ is harmonic inside $y>0$, since it is equal to the imaginary part of the holomorphic function $z \mapsto z^{1/2}$, $z=x_n+iy_n$.
Moreover, by a direct computation it is easily checked that $G$ satisfies
$$
|\nabla_x G|^2 + (G_y)^2 -\frac{G G_y}{y}=0\qquad \forall \,(x,y) \in\R^n \times \R^+.
$$
Thanks to this fact, since $\bar H=G^{2(1-s)}=G^{1+a}$, we get
$$
L_{-a}\bar H=L_{-a}\bigl(G^{1+a}\bigr)=(1+a)G^a \Delta_{x,y}G+(1+a)a\, G^{a-1} \left[ |\nabla_x G|^2 + (G_y)^2 -\frac{G G_y}{y}\right]=0,
$$
as desired.
}
Let $\l_1$ denote the eigenvalue corresponding to $\bar h$, so that
$$
\inf_{h \in H^{1/2}(\S^{n+1}_0),\,h=0 \text{ on } \S^{n+1}_{0,+} }
\frac{\int_{\S^{n+1}_+}|\n_\theta h|^2y^{-a}\,d\sigma}{\int_{\S^{n+1}_+} h^2y^{-a}\,d\sigma}=-\l_1.
$$
In order to explicitly compute $\l_1$, we observe that
$\div_\theta(y^{-a} \nabla_\theta \bar h)=\l_1\bar h$.
In particular, evaluating the above identity at the point $(0,1)\in \R^n\times \R^+$ we obtain
$$
\Delta_\theta \bar h(0,1)=\l_1 \bar h(0,1).
$$
We now write the equation $L_{-a} \bar H=0$ in spherical coordinates:
$$
\Delta_r \bar H+\frac{n}{r} \bar H_r +\frac{1}{r^2}\Delta_\theta \bar H-\frac{a}{y}\bar H_y=0
$$
Evaluating the above expression at $(0,1)\in \R^n\times \R^+$ and recalling that $\bar H=r^{1-s}\bar h$, we get
\begin{align*}
0&=\Delta_r \bar H(0,1) + n \bar H_r(0,1) + \Delta_\theta \bar H(0,1) - a \bar H_r(0,1)\\
&=-(1-s)s \,\bar h(0,1) + (1-s)(n-a)\,\bar h(0,1) + \Delta_\theta \bar h(0,1).
\end{align*}
Hence
$$
\Delta_\theta \bar h(0,1)=- (1-s)(n-a-s)\,\bar h(0,1)=- (1-s)(n-1+s)\,\bar h(0,1).
$$
which gives $\l_1=- (1-s)(n-1+s)$ as desired.
\end{proof}

To simplify notation, we use the variable $z$ to denote a point $(x,y)\in \R^n\times \R^+$.

\begin{lemma}
\label{lemma:monotonicity}
Let $w$ and $r_0$ be as above, set
$B_r^+:=\{z \in \R^n\times \R^+\,:\, |z| <r\}$, and define
$$
\var(r):=\frac{1}{r^{2(1-s)}}\int_{B_r^+} \frac{|\n w(z)|^2y^{-a}}{|z|^{n-1-a}}\,dz\qquad \forall \,r\leq 1.
$$
There exists a constant $C''$, depending on $C_0$, $\|\Df \psi\|_{C_x^{1-s}(\R^n)}$, $\|v-\psi\|_{L^\infty(\R^n)}$, and $\|\Df v\|_{L^\infty(\R^n)}$ only, such that
$$
\var(r)\leq C''\left[1+r^{2\a +\d_\a -a-1}\right]\qquad \forall \,r\leq 1.
$$
\end{lemma}

\begin{proof}
First of all, we show that $\var(1)$ is universally bounded, so that in particular
$\var(r)$ is well-defined for all $r \in (0,1]$.

Set $\var_\e(r):=\frac{1}{r^{2(1-s)}}\int_{B_r^+ \cap \{y>\e\}} \frac{|\n w(z)|^2y^{-a}}{|z|^{n-1-a}}\,dz$.
By the monotone convergence theorem, it suffices to estimate $\liminf_{\e\to 0}\var_\e(1)$.
Let $\chi:\R^n \to [0,1]$ be a smooth compactly supported function such that $\chi\equiv 1$ on $B_1\subset \R^n$. Then 
$$
\var_\e(r) \leq \int_\e^1 \int_{\R^n} \frac{|\n w(z)|^2y^{-a}}{|z|^{n-1-a}}\chi(x)\,dx\,dy.
$$
Since $L_{-a} w=0$ we have $L_{-a}(w^2) =
2w L_{-a} w + 2|\n w|^2y^{-a} = 2|\n w|^2y^{-a}$, which implies that the right hand side is equal to
\begin{align*}
 \int_\e^1 \int_{\R^n} \frac{L_{-a}(w^2)}{2|z|^{n-1-a}}\chi(x)\,dx\,dy&=
 \int_\e^1 \int_{\R^n} \nabla(w^2)\cdot \nabla \biggl(\frac{1}{2|z|^{n-1-a}}\biggr)y^{-a}\chi(x)\,dx\,dy\\
& +\int_\e^1 \int_{\R^n} \nabla_x(w^2)\cdot \nabla_x \chi(x) \frac{y^{-a}}{2|z|^{n-1-a}}\,dx\,dy\\
 &\quad+ 
  \int_{\R^n} (w^2)_y\frac{y^{-a}}{2|z|^{n-1-a}}\chi(x)\,dx\bigg|_{y=\e}^{y=1}.
  \end{align*}
Integrating by parts again the first two terms in the right hand side,
and using that $L_{-a} \frac{1}{|z|^{n-1-a}}=C_{n,a} \delta_{(0,0)}$, 
we find that the above expression coincides with
  \begin{multline*}
 -\int_\e^1 \int_{\R^n} w^2 \nabla_x\chi(x)\cdot \nabla_x \biggl(\frac{1}{|z|^{n-1-a}}\biggr)y^{-a}\,dx\,dy -\int_\e^1 \int_{\R^n} w^2\Delta_x \chi(x) \frac{y^{-a}}{2|z|^{n-1-a}}\,dx\,dy\\
 +  \int_{\R^n} w^2 \p_y \biggl(\frac{1}{2|z|^{n-1-a}}\biggr)y^{-a}\chi(x)\,dx\bigg|_{y=\e}^{y=1}
+  \int_{\R^n} (w^2)_y\frac{y^{-a}}{2|z|^{n-1-a}}\chi(x)\,dx\bigg|_{y=\e}^{y=1}.
  \end{multline*}
  Now, since $\chi\equiv 1$ inside $B_1$,
  the first two terms above are immediately seen to be bounded.
  Concerning the last two terms, the integrals evaluated at $y=1$ are clearly finite (and universally bounded),
  since $w$ is smooth for $y>0$. Finally, we apply Lemma \ref{lemma:prelim for monot}
  to estimate the integrals at $y=\e$,
and we obtain
  $$
  \var(1) = \liminf_{\e\to 0} \var_\e(1) \leq C_\var,
  $$
  for some constant $C_\var$ depending on $C_0$, $\|\Df \psi\|_{C_x^{1-s}(\R^n)}$, $\|v-\psi\|_{L^\infty(\R^n)}$, and $\|\Df v\|_{L^\infty(\R^n)}$ only. Observe that, as a consequence of the  fact that $\var(1)$ is finite (i.e., $ \frac{|\n w(z)|^2y^{-a}}{|z|^{n-1-a}}$ is integrable over $B_1^+$), we deduce that $\var_\e(r)\to \var(r)$ locally uniformly over $(0,1]$.

Now that we have proved that $\var(r)$ is well-defined, we want to estimates from below its
derivative. Again, we will do our computations with $\var_\e$, and then we let $\e\to 0$\footnote{
The proof of the monotonicity formula may look a bit tedious, since we always prove the result
at the $\epsilon$ level, and then we show that one can take the limit as $\e\to 0$.
Let us point out that this level of precision is actually needed: indeed, assume that 
we had chosen a different operator $L_b$
($b \in (-1,1)$) to define $w$ in \eqref{eq:def w}, and we defined $\var(r)$ replacing $-a$ by $b$
(changing, of course, the value of $s$ correspondingly). 
Then, if one does a ``formal'' proof of the monotonicity formula, one would obtain (at least in the stationary case, so that $w(x,0)=\Df v(x)$)
that Lemma \ref{lemma:monotonicity}
is true with $b$ in place of $-a$, and this would imply a false H\"older regularity for $w$
(since we know that $w$ should be only $C^{1-s}$). The fact that we have chosen the ``right'' operator $L_{-a}$ to define $w$ has played a key role in the proof of Lemma  \ref{lemma:prelim for monot}, which is now providing to us some fundamental estimates, which are needed to give a rigorous proof of the monotonicity formula.}.

Let us assume $r>\e$, and 
split $\p \left( B_r^+\cap \{y>\e\}\right)$ as the union of $\p B_r^+\cap \{y=\e\}$
and $\p B_r^+\cap \{y>\e\}$.
Using again that $L_{-a} \frac{1}{|z|^{n-1-a}}=C_{n,a} \delta_{(0,0)}$ and recalling that $a=1-2s$, we easily get
\begin{align*}
\var_\e'(r)&=-\frac{1-s}{r^{1+2(1-s)}} \int_{B_r^+ \cap \{y>\e\}} \frac{L_{-a}(w^2)}{|z|^{n-1-a}}\,dz+
\frac{1}{r^{n}} \int_{\p B_r^+\cap \{y>\e\}} |\n w(z)|^2y^{-a}\,d\sigma\\
&=-\frac{2(1-s)}{r^{1+2(1-s)}}
\int_{\p (B_r^+\cap \{y>\e\})}-w\,\n w\cdot \nu\, \frac{y^{-a}}{|z|^{n-1-a}}\,d\sigma\\
&\quad-\frac{1-s}{r^{1+2(1-s)}} \int_{B_r^+\cap \{y>\e\}} \n (w^2)\cdot \n \left( \frac{1}{|z|^{n-1-a}}\right)y^{-a}\,dz
+\frac{1}{r^{n}} \int_{\p B_r^+\cap \{y>\e\}} |\n w(z)|^2y^{-a}\,d\sigma\\
&= \frac{2(1-s)}{r^{1+2(1-s)}} \int_{\p (B_r^+\cap \{y>\e\})}-w\,\n w\cdot \nu\, \frac{y^{-a}}{|z|^{n-1-a}}\,d\sigma\\
&\quad+\frac{1-s}{r^{1+2(1-s)}}  \int_{\p (B_r^+\cap \{y>\e\})}
 -w^2 \n \left( \frac{1}{|z|^{n-1-a}}\right)
\cdot \nu\,y^{-a}\,d\sigma\\
&\quad+\frac{1}{r^{n}} \int_{\p B_r^+\cap \{y>\e\}} |\n w(z)|^2y^{-a}\,d\sigma\\
&= -\frac{2(1-s)}{r^{n+1}} \int_{\p B_r^+\cap \{y>\e\}}w\,w_r\,y^{-a}\,d\sigma
+\frac{1-s}{r^{1+2(1-s)}} 
\int_{B_r^+ \cap\{y=\e\}}(w^2)_y\,\frac{y^{-a}}{|z|^{n-1-a}}\,d\sigma\\
&\quad-\frac{(1-s)(n-1-a)}{r^{n+2}} \int_{\p B_r^+\cap \{y>\e\}} w^2y^{-a} \,d\sigma\\
&\quad+\frac{1-s}{r^{1+2(1-s)}}
\int_{B_r^+ \cap\{y=\e\}} w^2 \p_y \left( \frac{1}{|z|^{n-1-a}}\right)\,y^{-a}\,d\sigma
+\frac{1}{r^{n}} \int_{\p B_r^+\cap \{y>\e\}} |\n w(z)|^2y^{-a}\,d\sigma.
\end{align*}
Thanks to Lemma \ref{lemma:prelim for monot}, we can estimate from below both the second and the last but one term in the last expression.
So, letting $\e \to 0$ and using that $\var_\e \to \var$ locally uniformly, we deduce that the distributional derivative $D_r\var$ of $\var$ is bounded from below
by
\begin{multline*}
-\frac{2(1-s)}{r^{n+1}} \int_{\p B_{r,+}}w\,w_r\,y^{-a}\,d\sigma-C\,r^{\a+1+a}\\
-\frac{(1-s)(n-1-a)}{r^{n+2}} \int_{\p B_{r,+}} w^2y^{-a} \,d\sigma
+\frac{1}{r^{n}} \int_{\p B_{r,+}} |\n w(z)|^2y^{-a}\,d\sigma ,
\end{multline*}
for some universal constant $C$.
Now, by Schwartz's inequality the first term in the above expression can be estimated from below by
$$
-\frac{1}{r^{n}} \int_{\p B_{r,+}} (w_r)^2y^{-a}\,d\sigma- \frac{(1-s)^2}{r^{n+2}}\int_{\p B_{r,+}} w^2y^{-a} \,d\sigma.
$$
Hence, recalling that $|\n w(z)|^2=(w_r)^2+\frac{1}{r^2}|\n_\theta w_\theta|^2$ and observing that
$n-1-a+1-s=n-1+s$, we obtain
$$
D_r \var\geq\frac{1}{r^{n+2}} \int_{\p B_{r,+}} |\n_\theta w(z)|^2y^{-a}\,d\sigma
-\frac{(1-s)(n-1+s)}{r^{n+2}} \int_{\p B_{r,+}} w^2y^{-a} \,d\sigma
-C\,r^{\a+1+a}.
$$
Consider now the function $\bar W:=(w-r^{\a+\d_\a})^-$. Then $|\n_\theta \bar W(z)|^2\leq |\n_\theta w(z)|^2$.
Moreover, by Lemma \ref{lem:apply monot}, $\bar W$ is admissible for the eigenvalue problem in Lemma \ref{lem:eigen}.
Hence
\begin{align*}
D_r \var&\geq \frac{1}{r^{n+2}} \int_{\p B_{r,+}} |\n_\theta \bar W(z)|^2y^{-a}\,d\sigma
-\frac{(1-s)(n-1+s)}{r^{n+2}} \int_{\p B_{r,+}} \bar W^2 y^{-a}\,d\sigma\\
&\qquad+ \frac{(1-s)(n-1+s)}{r^{n+2}} \int_{\p B_{r,+}} \bigl[(w-\bar W)^2 +2\bar W(w-\bar W)\bigr]\,y^{-a}\,d\sigma -C\,r^{\a+1+a}\\
&\geq \frac{(1-s)(n-1+s)}{r^{n+2}} \int_{\p B_{r,+}} \bigl[(w-\bar W)^2 +2\bar W(w-\bar W)\bigr]\,y^{-a}\,d\sigma -C\,r^{\a+1+a}.
\end{align*}
Since $|\bar W|\leq |w|\leq C\,r^\a$ and $|w-\bar W|\leq r^{\a +\d_\a}$ we obtain
$$
D_r \var\geq -C\,r^{2\a +\d_\a -a-2} - Cr^{\a+1+a},
$$
which integrated over $[r,1]$ gives
$$
\var (r)\leq \var(1)+ C\,r^{2\a +\d_\a -a-1} + C\qquad \forall \,r \leq 1.
$$
(Recall that $1+a>0$.)
Since $\var(1)$ is universally bounded, this concludes the proof.
\end{proof}

We are now ready to prove the optimal decay rate around free boundary points.
\begin{proposition}
\label{prop:C1s}
There exists a constant $\tilde C>0$, depending on $C_0$, $\|\Df \psi\|_{C_x^{1-s}(\R^n)}$, $\|v-\psi\|_{L^\infty(\R^n)}$, and $\|\Df v\|_{L^\infty(\R^n)}$ only,
such that \eqref{eq:C1s} holds.
\end{proposition}
\begin{proof}
Define $w_\e=w \ast \rho_\e$, where $\rho_\e=\rho_\e(x)$ is a smooth convolution kernel.
Since $L_{-a}$ commutes with convolution in the $x$ variable, $w_\e$ is $L_{-a}$-harmonic on $\R^n\times \R^+$.
Moreover, by \eqref{eq:bound below w}, $w_\e(x,y)-w_\e(x,0) \geq-\frac{nC_0}{1+a}y^{1+a}$.

Set $\bar W_\e:=(w_\e-r^{\a+\d_\a})^+$. Then it is easily seen that  $\bar W_\e$ is $L_{-a}$-subharmonic for $y>0$, and
$\bar W_\e(x,y)-\bar W_\e(x,0) \geq-\frac{nC_0}{1+a}y^{1+a}$.
Consider now the function
$$
\tilde w_\e(x,y):=
\bar W_\e(x,|y|) + \left(1+\frac{nC_0}{1+a}\right)\,|y|^{1+a}\qquad \text{on }\R^n \times \R.
$$
We observe that $\tilde w_\e$ is $L_{-a}$-subharmonic outside $\{y=0\}$.
Moreover, since $\tilde w_\e(x,y)-\tilde w_\e(x,0) \geq |y|^{1+a}$
and $\tilde w_\e$ is smooth in the $x$ variable, we deduce that it is $L_{-a}$-subharmonic on the whole $\R^n\times \R$.
Letting $\e \to 0$ we obtain that
$$
\tilde w(x,y):=
\bigl(w(x,|y|)-r^{\a+\d_\a}\bigr)^++ \left(1+\frac{nC_0}{1+a}\right)\,|y|^{1+a}
$$
is globally $L_{-a}$-subharmonic. 
Thanks to Lemma \ref{lemma:monotonicity}, $\tilde w$ vanishes on more than half of the $n$-dimensional disc $B_r \times \{0\}$.
So we can apply a weighted Poincar\'e inequality (see \cite{fabkenser}) and the definition of $\var$ (see Lemma \ref{lemma:monotonicity}) to get
\begin{align*}
\int_{B_r^+} (\tilde w)^2y^{-a}\,dz&\leq C\,r^2\int_{B_r^+} |\nabla \tilde w|^2y^{-a}\,dz\\
&\leq
C\,r^2\biggl[\int_{B_r^+} |\nabla w|^2y^{-a}\,dz + \, r^{n+1+a}\biggr]\\
&\leq C\, r^{n+2} \bigl[\var(r)+r^{1+a}\bigr] \leq C\, r^{n+2} \bigl[\var(r)+1\bigr] \qquad \forall \,r \leq 1,
\end{align*}
which combined with the $L_{-a}$-subharmonicity of $\tilde w$ and Lemma \ref{lemma:monotonicity} gives
\begin{align*}
\sup_{B_{r/2}^+} (\tilde w)^2 &\leq \frac{C}{r^{n+1-a}}\int_{B_r^+} (\tilde w)^2y^{-a}\,dz \\
&\leq C\,r^{1+a} \bigl[\var(r)+1\bigr]
\leq  C\,\left[r^{1+a}+r^{2\a+\delta_\a}\right].
\end{align*}
Hence, since $1+a=2(1-s)$ we obtain
$$
\sup_{B_{r/2}^+} w \leq C\biggl[\sup_{B_{r/2}^+} \tilde w +r^{\a+\d_\a}+r^{1+a}\biggr] \leq C\left[r^{1-s}+r^{\a+\delta_\a/2}\right] \qquad \forall \,r \leq 1.
$$
Since the above bound holds at every free boundary point, by the very same argument as in the proof of Corollary \ref{cor:unif C1a v} we obtain that
$\|w\|_{C_{x}^{\beta_\a}(\R^n)} \leq C$, where $\beta_\a=\beta_a(s):= \min\{\a+\delta_\a/2,(1-s)\}$.
Observe now that, by the formula for $\delta_\a$ provided in Lemma
\ref{lem:apply monot},
given $\alpha_0>0$ there exists $\delta_0>0$ such that
$\delta_\a\geq \delta_0>0$ for $\alpha \in [\a_0,1-s]$.
Hence, by iterating the above argument $k$ times we get
$$
\sup_{B_{r/2}^+} \tilde w \leq C\left[r^{1-s}+r^{\a+k\delta_0/2}\right] \qquad \forall \,r \leq 1,
$$
and after finitely many iterations we obtain \eqref{eq:C1s}.
\end{proof}

Arguing as in the proof of Corollary \ref{cor:unif C1a v}, \eqref{eq:C1s} gives:
\begin{corollary}
\label{cor:C1s}
There exists a constant $\bar C''>0$, depending on $C_0$, $\|\Df \psi\|_{C_x^{1-s}(\R^n)}$, $\|v-\psi\|_{L^\infty(\R^n)}$, and $\|\Df v\|_{L^\infty(\R^n)}$ only, such that 
$$
\|\Df v \chi_{\{v=\psi\}}\|_{C_{x}^{1-s}(\R^n)}\leq \bar C''.
$$
\end{corollary}

\subsection{Almost optimal regularity of solutions to the parabolic fractional obstacle problem}
\label{sect:C1s u}
Let $u$ be a solution of \eqref{eq:frac heat}, with $\psi \in C^{2}(\R^n)$ satisfying
assumptions (A1)-(A2) of Subsection \ref{sect:general C1a}.
Let us remark that $u(0)=\psi$, while for all $t>0$ we can apply the results of the previous subsections with $v=u(t)$. Hence
Corollary \ref{cor:C1s} gives:
\begin{proposition}
\label{prop:unif C1a rhs}
Let $u,\psi$ be as above. Then there exists a constant $\bar C_T>0$, depending on $T$, $\|D^2\psi\|_{L^\infty(R^n)}$, and
$\|\Df \psi\|_{C_x^{1-s}(\R^n)}$ only, such that
%the following hold for all $t \geq 0$:
%\begin{enumerate}
%\item[(i)] For every $x_0 \in \p \{u(t)=\psi\} $
%$$
%%\sup_{B_r(x_0)}|u(t)-\psi| \leq C\, r^{1+\a},\quad
%\sup_{B_r(x_0)}|\Df u(t)| \leq C\,
%r^{1-s}\qquad \forall \,r \leq \delta.
%$$
%$\Df u(t) \chi_{\{u(t)=\psi\}} \in C_{{\rm loc}}(\R^n)$, with
$$
\sup_{t\in [0,T]} \|\Df u(t) \chi_{\{u(t)=\psi\}}\|_{C_{x}^{1-s}(\R^n)}\leq \bar C_T.
$$
%\end{enumerate}
\end{proposition}
\begin{proof}
As explained at the beginning of Subsection \ref{sect:general C1a},
$u(t)$ satisfies assumptions (A5)-(A6) of Subsection \ref{sect:general C1a}
for every $t>0$ (see Lemma \ref{lem:sign frac lapl}). Moreover, (A3)-(A4) follow from Lemma \ref{lemma:Linfty bound frac heat} (since
$\|u(t)-\psi\|_{L^\infty(\R^n)} \leq T \|u_t\|_{L^\infty([0,T]\times \R^n)} \leq T \|\Df \psi\|_{L^\infty(\R^n)}$).
Hence the result is an immediate consequence of Corollary \ref{cor:C1s} applied to $v=u(t)$ for any $t>0$.
(For $t=0$ the result is trivial since $u(0)=\psi$.)
\end{proof}

Now, we want to exploit the fact that $u$ solves the parabolic equation
\begin{equation}
\label{eq:parabolic rhs}
u_t +\Df u=\bigl(\Df u\bigr) \chi_{\{u=\psi\}}\qquad \text{on }(0,T]\times \R^n.
\end{equation}
Thanks to Proposition \ref{prop:unif C1a rhs}, the right-hand side
of \eqref{eq:parabolic rhs} belongs to $L^\infty([0,T];C_{x}^{1-s}(\R^n))$,
which by parabolic regularity implies
\begin{equation}
\label{eq:Df C0a}
u_t,\Df u \in L^\infty((0,T];C_{x}^{1-s-0^+}(\R^n)),
\end{equation}
see \eqref{eq:regularity ut holder}.
We now want to use \eqref{eq:parabolic rhs} and a bootstrap argument to obtain
the desired H\"older regularity in time. We start with a preliminary result:
\begin{lemma}
\label{lemma:bootstrap}
Let that $\Df u \, \chi_{\{u=\psi\}}\in L^\infty((0,T];C_{x}^{1-s}(\R^n))$.
Fix $\a \in \left[0,\min\left\{1,\frac{1-s}{2s}\right\}\right)$, and assume that:
\begin{enumerate}
\item[-] $u_t \in L^\infty((0,T]\times \R^n)$, $\Df u \in L^\infty((0,T];C_{x}^{1-s-0^+}(\R^n))$ if $\a=0$;
\item[-] $u_t \in C_{t,x}^{\a,1-s} ((0,T]\times \R^n)$, $\Df u \in L^\infty((0,T];C_{x}^{1-s}(\R^n))$  if $\a>0$.
\end{enumerate}
Then
$$
\Df u \, \chi_{\{u=\psi\}}\in 
\left\{
\begin{array}{ll}
C_{t,x}^{\frac{1-s}{1+s}-0^+,1-s} ((0,T]\times \R^n)&\text{if }\a=0,\\
{\rm \Lip}_tC_{x}^{1-s}((0,T]\times \R^n)& \text{if }\a>0,\,s < 1/3,\\
{\rm \logLip}_tC_{x}^{1-s}((0,T]\times \R^n)& \text{if }\a>0,\,s = 1/3,\\
C_{t,x}^{(1+\a)\frac{1-s}{1+s},1-s}((0,T]\times \R^n)& \text{if }\a>0,\,s >1/3,
\end{array}
\right.
$$
with a uniform bound.

Moreover, for $s>1/3$ the function
$$
\a \mapsto \Phi(\a):=(1+\a)\frac{1-s}{1+s}
$$
is strictly increasing on $\left[0,\frac{1-s}{2s}\right)$, and $\Phi\left(\frac{1-s}{2s}\right)=\frac{1-s}{2s}$.
\end{lemma}
\begin{proof}
We need to estimate
\begin{equation}
\label{eq:holder time rhs}
|\Df u(t,x)\chi_{\{u(t)=\psi\}}-\Df u(s,x)\chi_{\{u(s)=\psi\}}|,\qquad 0< s<t.
\end{equation}
Since $\{u(t)=\psi\}\subset \{u(s)=\psi\}$ for $s<t$, we can assume that
$x \in \{u(s)=\psi\}$ (otherwise the above expression vanishes and there is nothing to prove).
Moreover,
since $\Df u$ vanishes on the free boundary, if $x \in\{u(s)=\psi\}\setminus \{u(t)=\psi\}$
we can alway find a time $\tau \in(s,t)$ such that
$$
\Df u(\tau,x)\chi_{\{u(\tau)=\psi\}}=\Df u(t,x)\chi_{\{u(t)=\psi\}}\quad \text{and}\quad x \in \p\{u(\tau)=\psi\}.
$$
Then, if we can estimate \eqref{eq:holder time rhs} with $\tau$ in place of $t$,
then we will also get the desired bound by simply replacing $\tau$ with $t$.
Hence, we only need to consider the case $x \in \{u(t)=\psi\}$.

We have to estimate $|\Df u(t,x)-\Df u(s,x)|$.
Let $\phi$ be a smooth non-negative cut-off function 
supported in $B_1$ such that $\int_{\R^n}\phi=1$, set $\phi_r(x):=\frac{1}{r^n}\phi\left(\frac{x}{r}\right)$, and compute
\begin{equation}
\label{eq:triangle inequality}
\begin{split}
|\Df u(t,x)-\Df u(s,x)| &\leq \left|\int_{\R^n}[\Df u(t,x) - \Df u(t,z)]\phi_r(x-z) \,dz\right|\\
&\qquad+\left|\int_{\R^n}[\Df u(t,z)-\Df u(s,z)]\phi_r(x-z) \,dz\right|\\
&\qquad+ \left|\int_{\R^n}[\Df u(s,x) - \Df u(s,z)]\phi_r(x-z) \,dz\right|,
\end{split}
\end{equation}
We now distinguish between two cases:\\
\textit{$\bullet$ $\a=0$.}
Thanks to the $C_x^{1-s-0^+}$-regularity of $\Df u$
and the fact that $\supp\, \phi_r \subset B_r$,
the first and the third term in the right hand side of \eqref{eq:triangle inequality} are bounded by $C\,r^{1-s-0^+}$.
For the second term, we integrate $\Df$ by parts, and using that
$\|\Df \phi_r\|_{L^1(\R^n)}=\|\Df \phi\|_{L^1(\R^n)}/r^{2s} \sim 1/r^{2s}$ 
and that $u$ is Lipschitz in time (Corollary \ref{cor:Lip}), we get
$$
|\Df u(t,x)-\Df u(s,x)| \leq C\left[r^{1-s-0^+} + \frac{(t-s)}{r^{2s}}\right].
$$
Choosing $r:=|t-s|^{1/(1+s)}$ we obtain
$$
|\Df u(t,x)-\Df u(s,x)| \leq C(t-s)^{\frac{1-s}{1+s}-0^+},
$$
as desired.\\
\textit{$\bullet$ $\a>0$.}
Arguing as above, the first and the third term in the right hand side of \eqref{eq:triangle inequality} are bounded by $C\,r^{1-s}$.
For the second one, we integrate again by parts and we estimate
\begin{align*}
\bigg|\int_{\R^n}&[\Df u(t,z)-\Df u(s,z)]\phi_r(x-z) \,dz\bigg|\\
&=\left|\int_{\R^n}[u(t,z)-u(s,z)]\Df \phi_r(x-z) \,dz\right|\\
&\leq \left|\int_{\R^n}\bigl[u(t,z)-u(s,z)- u_t(s,z)[t-s]\bigr]\Df \phi_r(x-z) \,dz\right|\\
&\qquad+ (t-s)\,\left|\int_{\R^n} |u_t(s,z)| \,|\Df \phi_r(x-z)| \,dz\right|.
\end{align*}
Since $\|\Df \phi_r\|_{L^1(\R^n)} \sim 1/r^{2s}$ and $u_t\in C_t^\a$, the
first term in the right hand side is bounded by $C\,\frac{(t-s)^{1+\a}}{r^{2s}}$.
For the second term, we observe that $u_t$ vanishes at $(t,x)\in \{u=\psi\}$, so by the $C_{x}^{1-s}$-regularity of $u_t$
we get
$$
\left|\int_{\R^n} |u_t(s,z)| \,|\Df \phi_r(x-z)| \,dz\right|
\leq C \left|\int_{\R^n} \min\left\{|x-z|^{1-s},1\right\} \,|\Df \phi_r(x-z)| \,dz\right|.
$$
We now remark that, since $\phi$ is compactly supported, $|\Df \phi(w)| \leq \frac{C}{|w|^{n+2s}}$ for $|w|$ large. So, there exists a constant
$C_\phi$, depending on $\phi$ only, such that
$$
|\Df \phi(w)|\leq  \frac{C_\phi}{1+ |w|^{n+2s}}\qquad \forall \,w \in \R^n,
$$
which by scaling gives
$$
|\Df \phi_r(w)|\leq  \frac{C_\phi}{r^{n+2s}+ |w|^{n+2s}}\qquad \forall \,w \in \R^n.
$$
Hence
\begin{align*}
\int_{\R^n} \min\left\{|w|^{1-s},1\right\}|\Df \phi_r(w)|\,dw &\leq C\int_{B_1} \frac{|w|^{1-s}}{r^{n+2s}+ |w|^{n+2s}}\,dw
+ C\int_{\R^n\setminus B_1} \frac{1}{|w|^{n+2s}}\,dw\\
& \leq \frac{C}{r^{n+2s}} \int_{B_r}|w|^{1-s}\,dw + C \int_{B_1\setminus B_r} |w|^{1-3s - n}\,dw + C,
\end{align*}
which implies
$$
\int_{\R^n} \min\left\{|w|^{1-s},1\right\}|\Df \phi_r(w)|\,dw \leq \left\{
\begin{array}{ll}
 C & \text{if }s<1/3;\\
C (1+|\log(r)|) & \text{if }s=1/3;\\
C (1+r^{1-3s}) & \text{if }s>1/3.
\end{array}
\right.
$$
All in all, we have obtained
$$
|\Df u(t,x)-\Df u(s,x)| \leq C\biggl[r^{1-s} + \frac{(t-s)^{1+\a}}{r^{2s}}+(t-s)\left\{\begin{array}{ll}
 C & \text{if }s<1/3;\\
C (1+|\log(r)|) & \text{if }s=1/3;\\
C (1+r^{1-3s}) & \text{if }s>1/3.
\end{array}
\right.\biggr]
$$
Choosing $r:=(t-s)^{(1+\a)/(1+s)}$, the above estimates give:
\begin{enumerate}
\item[-] $s<1/3$:
$|\Df u(t,x)-\Df u(s,x)| \leq C(t-s).$
\item[-] $s=1/3$:
$|\Df u(t,x)-\Df u(s,x)| \leq C(t-s)\bigl[1+|\log(t-s)|\bigr].$
\item[-] $s>1/3$: Since $\a \leq (1-s)/2s$ by assumption, we have
$\a \leq \frac{(1-s)(1+\a)}{1+s} \leq 1+\frac{(1-3s)(1+\a)}{1+s},$ so
\begin{align*}
|\Df u(t,x)-\Df u(s,x)| &\leq C\left[(t-s)^{\frac{(1+\a)(1-s)}{1+s}} + (t-s)+(t-s)^{1+\frac{(1-3s)(1+\a)}{1+s}}\right]\\
& \leq C\,(t-s)^{\frac{(1+\a)(1-s)}{1+s}}.
\end{align*}
\end{enumerate}
\end{proof}

Thanks to the above lemma, we can use \eqref{eq:parabolic rhs} and a bootstrap argument to prove 
our main regularity result.
\begin{proof}[Proof of Theorem \ref{thm:main}]
The global Lipschitz regularity of $u$ in space-time follows from Corollary \ref{cor:Lip}.

By Proposition \ref{prop:unif C1a rhs} and
\eqref{eq:Df C0a}, we can apply Lemma \ref{lemma:bootstrap} with $\a=0$ to deduce that
the right hand side of \eqref{eq:parabolic rhs} belongs to $C_{t,x}^{\frac{1-s}{1+s}-0^+,1-s}((0,T]\times \R^n) $.
Hence, since $2s<1+s$, by the parabolic regularity theory for $\p_t+\Df$ (see \eqref{eq:regularity holder multipliers}) we get
$u_t,\Df u \in C_{t,x}^{\frac{1-s}{1+s}-0^+,1-s}((0,T]\times \R^n)$.
We now apply Lemma \ref{lemma:bootstrap} with $\a>0$, and 
we distinguish between two cases:
\begin{enumerate}
\item[-] $s\leq 1/3$: In this case we get
$$
\Df u \, \chi_{\{u=\psi\}}\in {\rm \logLip}_tC_{x}^{1-s}((0,T]\times \R^n),
$$
so by \eqref{eq:regularity holder multipliers} and \eqref{eq:regularity holder 2}
we get $\Df u\in \logLip_tC_x^{1-s}((0,T]\times\R^n)$, and we conclude by using $u_t=\Df u \, \chi_{\{u=\psi\}}-\Df u$.
\item[-] $s>1/3$: Lemma \ref{lemma:bootstrap} gives
$$
\Df u \, \chi_{\{u=\psi\}}\in C_{t,x}^{\Phi\left(\frac{1-s}{1+s}-0^+\right),1-s}((0,T]\times \R^n),
$$
which by \eqref{eq:regularity holder multipliers} implies $u_t,\Df u \in C_{t,x}^{\Phi\left(\frac{1-s}{1+s}-0^+\right),1-s}((0,T]\times \R^n)$
(recall that $\Phi(\a)<\frac{1-s}{2s}$ if $\a <\frac{1-s}{2s}$). Hence, we can use iteratively Lemma \ref{lemma:bootstrap} and \eqref{eq:regularity holder multipliers}
to get
$$
u_t,\Df u\in C_{t,x}^{\Phi^n\left(\frac{1-s}{1+s}-0^+\right),1-s}((0,T]\times \R^n),
$$
which together with \eqref{eq:regularity holder 1} implies
$$
\Df u\in C_{t,x}^{\frac{1-s}{2s},1-s}((0,T]\times \R^n).
$$
Finally, since $\Phi^n\left(\frac{1-s}{1+s}-0^+\right) \nearrow \frac{1-s}{2s}$ as $n \to \infty$, we obtain
$$
u_t\in C_{t,x}^{\frac{1-s}{2s}-0^+,1-s}((0,T]\times \R^n),
$$
as desired.
\end{enumerate}
\end{proof}

\section{Extension to more general equations}
\label{sect:general eq}
In this section we give a brief informal description of the main modifications needed
to extend the regularity result of Theorem \ref{thm:main}
to solutions of \eqref{eq:tilde LL}, at least when $s>1/2$. Our aim is only to point out the major differences with respect to the model case
\eqref{eq:frac heat} treated above, and to explain how to handle them.
There will be however to attempt to state a proper theorem, as this would need a careful
analysis of the assumptions needed on $\psi,\K$ (for instance,
since the operators are non-local, in all the estimates one should take care of the contribution
coming from infinity). We plan to address this issue in a future work. \\

Assume that $\psi:\R^n \to \R^+$ is a smooth globally Lipschitz function, $b \in \R^n$ a vector,
$r\geq 0$ is a constant and
$\K$ a (smooth) non-local translation-invariant elliptic operator of lower order with respect to $\Df$, i.e., there exists $\k\in (0,1)$ such that
$$
[\K \var ]_{C^{\k}_{{\rm loc}}(\R^n)} \lesssim \|\Df \var\|_{L^\infty(\R^n)}\qquad \forall \, \var \in C_c^\infty(\R^n).
$$
We consider $u:[0,T]\times \R^n \to \R$ a (continuous) viscosity solution to the obstacle problem
\begin{equation}
\label{eq:general frac heat visc}
\left\{
\begin{array}{l}
\min\{ -u_t  +ru + b\cdot \nabla u - \Df u - \K u ,u-\psi\}=0 \quad \text{on }(0,T]\times \R^n,\\
u(0)=\psi,
\end{array}
\right.
\end{equation}
When $s >1/2$, existence and uniqueness of such a solution
follows again by standard results on  obstacle problems.

Let us now analyze the properties of solutions, as we did before for \eqref{eq:frac heat}.\\

\textit{$\bullet$ Basic properties.} We proceed as in Section \ref{sect:prelim}. First of all, 
as in Lemma \ref{lemma:comparison} one can approximate solutions to \eqref{eq:general frac heat visc} using a penalization method.
In this way, all the results of Section \ref{sect:prelim} still hold true:
\begin{enumerate}
\item[-] $u(t,\cdot)$ is globally Lipschitz (see Lemma \ref{lemma:basic});
\item[-] $u(t,\cdot)$ is uniformly semiconvex (see Lemma \ref{lemma:basic});
\item[-] $u_t$ is globally bounded (see Lemma \ref{lemma:Linfty bound frac heat});
\item[-] $\Df u + \K u - r u- b \cdot \nabla u$ is globally bounded (see Lemma \ref{lemma:Linfty bound frac heat}).
\end{enumerate}
In particular, by elliptic regularity for the fractional Laplacian, the $L^\infty$-bound on $\Df u + \K u(t) - r u(t)- b \cdot \nabla u(t)$
gives 
\begin{enumerate}
\item[-] $u \in L^\infty([0,T],C_{{\rm loc}}^{2s-0^+}(\R^n))$.
\end{enumerate}
Hence, since $s>1/2$ and $\K u$ is of order $\leq 2s-\k$, there exists
$\gamma=\gamma(s,\k)>0$ such that
\begin{equation}
\label{eq:R}
R:=- \K u+ru +b \cdot \nabla u \in L^\infty([0,T],C_{{\rm loc}}^{\g}(\R^n)).
\end{equation}

\textit{$\bullet$ $C_{x}^{\a}$-decay for $\Df u(t)$.}
In this setting, we have:
\begin{enumerate}
\item[-] $\Df u(t) -R(t)< 0$ inside the open set $\{u(t)>\psi\}$;
\item[-] $\Df u(t)-R(t)\geq 0$ a.e. on $\{u(t)=\psi\}$.
\end{enumerate}
Now, fixed $t>0$ and given a free boundary point $x_0 \in \p\{u(t)=\psi\}$,
we consider the $L_a$-harmonic function $v(x,y):=u(t,x,y) -\frac{R(t,x_0)}{1-a}\,y^{1-a}$,
where $u(t,x,y)$ is the $L_a$-harmonic extension of $u(t)$. Moreover, as in Subsection \ref{sect:general C1a}
we consider the function
$$
\tilde v(x,y):=v(x,y)-\psi(x).
$$
Since $u_t$ is globally bounded and $v(\cdot,0)$ is semiconvex, all estimates (B1)-(B2) and (B4)-(B5) of Subsection \ref{sect:general C1a}
still hold true, while (B3) becomes
\begin{enumerate}
\item[(B3')] $\lim_{y \to 0^+} y^a\tilde v_y(x,y)\leq |R(t,x)-R(t,x_0)| \leq C|x-x_0|^\gamma$ for a.e. $x \in \L$,\\
$\lim_{y \to 0^+} y^a\tilde v_y(x,y) > -C|x-x_0|^\gamma$ for $x \in \R^n\setminus \L$.
\end{enumerate}
While the proof of Lemma \ref{lemma:deriv to fct} works with no modifications
under these assumptions,
for the proof of Proposition \ref{prop: C1a v} we remark that now we do not have
$\lim_{y \to 0^+} y^a\tilde v_y(x,y) \geq 0$ for $x \in \R^n\setminus \L$, which was used to apply Hopf's Lemma.
To overcome this difficulty, in the induction step from $k_0$ to $k_0+1$ one should
replace $v(x,y)$ with $v(x,y)+\frac{\|R(t)\|_{C^{\gamma}(B_1(x_0))}}{1-a} (4^{-k_0})^\gamma y^{1-a}$,
and the rest of the proof should go through with minor modifications.

Hence, one still gets $\sup_{B_r(x_0)}|u(t) - \psi| \leq Cr^{\a+2s}$ and
$$
[\Df u(t)-R(t)] \chi_{u(t)=\psi}=[\Df u(t)+ \K u(t) - r u(t)- b \cdot \nabla u(t)] \chi_{u(t)=\psi} \in C_{{\rm loc}}^{\a}(\R^n),
$$
for some universal exponent $\a\in(0,\gamma)$.\\

\textit{$\bullet$ Monotonicity formula and optimal spatial regularity.}
As in Subsection \ref{subsect:monot}, one would like to apply a monotonicity formula.
However, first of all one has to do a preliminary step:
using the equation
$$
u_t + \Df u=[\Df u+ R] \chi_{u=\psi} - R \in L^\infty([0,T];C_{{\rm loc}}^{\a}(\R^n))
$$
one deduces (thanks to a local variant of \eqref{eq:regularity holder 1}-\eqref{eq:regularity holder 2}) that
$\Df u \in L^\infty([0,T];C_{{\rm loc}}^{\a-0^+}(\R^n))$.
So, by elliptic regularity for the fractional Laplacian, $u \in C_{{\rm loc}}^{\a+2s-0^+}(\R^n)$, which gives
$$
R \in L^\infty([0,T],C_{{\rm loc}}^{\a+\gamma/2}(\R^n)).
$$
This allows considering $R$ as a lower order perturbation when applying the monotonicity formula.

Now, to apply the monotonicity formula around a free boundary point $x_0 \in \p\{u(t)=\psi\}$,
one should consider the function $w:\R^n\times \R^+\times \R$ obtained by solving
the Dirichlet problem
$$
\left\{
\begin{array}{l}
L_{-a}w=0,\\
w(x,0)=[\Df u(t,x,0)-R(t,x_0)] \chi_{\{u(t)=\psi\}}(x).
\end{array}
\right.
$$
Since $R(t) \in C_{{\rm loc}}^{\a+\gamma/2}(\R^n)$, we have $w \geq -C r^{\a+\g/2}$ on $B_r(x_0)$.
So, by the monotonicity formula one gets
$[\Df u+R(t,x)] \chi_{u(t)=\psi} \in C_{{\rm loc}}^{\beta_\a'}(\R^n)$, $\beta_\a':=\min\{\a+\d_\a',1-s\}$.
Then, one can iterate the above strategy, first using the parabolic regularity of $\p_t+\Df$ and the the elliptic regularity of $\Df$
to show that
$R \in L^\infty([0,T];C_{{\rm loc}}^{\b_\a'+\gamma}(\R^n)$, and then applying again the monotonicity formula.
In this way, after finitely many iterations we get
$$
[\Df u+ \K u - r u- b \cdot \nabla u] \chi_{u=\psi} \in L^\infty([0,T],C_{{\rm loc}}^{1-s}(\R^n)),\qquad R\in L^\infty([0,T],C_{{\rm loc}}^{1-s+\g}(\R^n)). 
$$

\textit{$\bullet$ Parabolic regularity and conclusion.}
Using Lemma \ref{lemma:bootstrap}, the argument in Subsection \ref{sect:C1s u}
applied to
$$
\p_t u +\Df u = [\Df u+ R] \chi_{u=\psi} - R
$$
allows to extend the regularity result in Theorem \ref{thm:main} (at least locally in space-time) to solutions of \eqref{eq:general frac heat visc}.

\appendix

\section{Regularity properties of the operator $\p_t+\Df$}

In this appendix we describe some important properties of the parabolic operator $\p_t+\Df$.\\

Let us first recall that fractional Laplacian works nicely in H\"older spaces: in
$f \in C^{\a}(\R^n)$ then $(-\D)^{-s}f \in C^{\a+2s}(\R^n)$, see
for instance \cite[Subsection 2.1]{silv}.

Analogously, the operator $\p_t + \Df$ works nicely in space-time H\"older spaces:
if $v_t+\Df v=f$ on $[0,T]\times \R^n$ and $v(0)$ is smooth (for our purposes, we can assume $v(0) \in C^2(\R^n)$, globally Lipschitz, and 
$\|D^2 v(0)\|_{L^\infty(\R^n)}+\|\Df v(0)\|_{C_x^{1-s}(\R^n)} <+\infty$),
by classical results on multipliers on H\"older spaces (see for instance \cite[Theorem 2.3]{madychriviere} and the proof of \cite[Theorem 3.1]{madychriviere}) we get
\begin{equation}
\label{eq:regularity holder multipliers}
\| v_t\|_{C_{t,x}^{\a,\b}((0,T]\times\R^n)}+\| \Df v\|_{C_{t,x}^{\a,\b}((0,T]\times\R^n)} \lesssim 1+ \|f\|_{C_{t,x}^{\a,\b}((0,T]\times\R^n)} \qquad \forall\,\a,\b \in (0,1)
\end{equation}
However, for our purposes, we also need to have some regularity estimates when $f$ is only bounded in time (but H\"older in space).

Let us observe that we can write the solution in terms of the fundamental solution $\G_s(t,x)$
of the fractional heat equation. More precisely, using Duhamel formula, we have
\begin{equation}
\label{eq:duhamel}
\begin{split}
-\Df v(t,x) &=v_t(t,x)-f(t,x)\\
&=-\G_s(t)\ast \Df v(0)+\int_0^t \int_{\R^n}\partial_t \G_s(t-\tau,x-y) [f(\tau,y)-f(\tau,x)]\,dy\,d\tau.
\end{split}
\end{equation}
We now claim that the following estimates hold (the proof of them is postponed to Subsection \ref{subsect:proof estimates} below):
\begin{equation}
\label{eq:infty regularity holder 1}
\left\|\int_0^t \int_{\R^n}\partial_t \G_s(t-\tau,x-y) [f(\tau,y)-f(\tau,x)]\,dy\,d\tau\right\|_{C_{t,x}^{\b/(2s),\b-0^+}((0,T]\times\R^n)} \lesssim \|f\|_{L^\infty((0,T];C_{x}^\b(\R^n))}
\end{equation}
if $\b<2s$,
\begin{equation}
\label{eq:infty regularity holder 2}
\left\|\int_0^t \int_{\R^n}\partial_t \G_s(t-\tau,x-y) [f(\tau,y)-f(\tau,x)]\,dy\,d\tau\right\|_{\logLip_tC_x^{\b-0^+}((0,T]\times\R^n)} \lesssim \|f\|_{L^\infty((0,T];C_{x}^\b(\R^n))}
\end{equation}
if $\b\geq 2s$.

Combining the above estimates with \eqref{eq:duhamel}, and using that $-\G_s(t)\ast \Df v(0)$ is smooth in space-time for $t>0$, we get
\begin{equation}
\label{eq:regularity holder 1}
\| \Df v\|_{C_{t,x}^{\b/(2s),\b-0^+}((0,T]\times\R^n)} \lesssim 1+ \|f\|_{L^\infty((0,T];C_{x}^\b(\R^n))} \qquad \forall\, \beta\in (0,2s)
\end{equation}
\begin{equation}
\label{eq:regularity holder 2}
\| \Df v\|_{\logLip_tC_x^{\b-0^+}((0,T]\times\R^n)}  \lesssim 1+ \|f\|_{L^\infty((0,T];C_{x}^\b(\R^n))}\qquad \forall\, \beta\in [2s,1)
\end{equation}
(At $t=0$ the time regularity may degenerate, due to the presence of
the term $\G_s(t)\ast \Df v(0)$.)
In particular, using that $v_t=f-\Df v$, we obtain
\begin{equation}
\label{eq:regularity ut holder}
\| v_t\|_{L^\infty((0,T];C_x^{\b-0^+}(\R^n))}+ \| \Df v\|_{L^\infty((0,T];C_x^{\b-0^+}(\R^n))}\lesssim 1+ \|f\|_{L^\infty((0,T];C_{x}^\b(\R^n))} \qquad \forall\, \beta\in (0,1)
\end{equation}

\subsection{Proof of \eqref{eq:infty regularity holder 1} and \eqref{eq:infty regularity holder 2}}
\label{subsect:proof estimates}
Let us recall that $t \in [0,T]$, with $T<+\infty$.

To prove \eqref{eq:infty regularity holder 1} and \eqref{eq:infty regularity holder 2}, we use
that the fundamental solution $\G_s(1,y)$ behaves like $\frac{1}{1+|y|^{n+2s}}$, which by scaling implies
\begin{equation}
\label{eq:fund sol 1}
\G_s(t,y) \sim \frac{t}{t^{\frac{n+2s}{2s}}+|y|^{n+2s}},\qquad |\p_t\G_s(t,y)| \lesssim \frac{1}{t^{\frac{n+2s}{2s}}+|y|^{n+2s}},
\end{equation}
\begin{equation}
\label{eq:fund sol 2}
|\p_{tt}\G_s(t,y)| 
\lesssim \frac{1}{t}\frac{1}{(t^{\frac{n+2s}{2s}}+|y|^{n+2s})}, \quad |\nabla_y\p_{t}\G_s(t,y)| 
\lesssim \frac{1}{|y|}\frac{1}{(t^{\frac{n+2s}{2s}}+|y|^{n+2s})}
\end{equation}
We will also make use of the following two basics estimates:\\
- There exists a constant $C>0$ such that, for all $h \in (0,1]$,
\begin{equation}
\label{estimate1}
\int_{\R^n}\frac{\min\{|z|^\b,1\}}{h^{\frac{n+2s}{2s}}+|z|^{n+2s}}\,dz
\leq C \Bigl(1+ h^{\b/(2s)-1}\Bigr).
\end{equation}
- There exists a constant $C>0$ such that, for all $h >0$,
\begin{equation}
\label{estimate2}
\int_0^t\frac{1}{(t-\tau)^{\frac{n+2s}{2s}}+h^{n+2s}}\,d\tau \leq C \min\left\{\frac{1}{h^n},\frac{1}{h^{n+2s}}\right\}.
\end{equation}
The proof of both is pretty simple.
For instance, to show \eqref{estimate1}, one splits the integral into three parts:
$$
\int_{B_{h^{1/(2s)}}}\frac{|z|^\b}{h^{\frac{n+2s}{2s}}+|z|^{n+2s}}\,dz \leq \frac{1}{h^{\frac{n+2s}{2s}}} \int_{B_{h^{1/(2s)}}}|z|^\b \,dz \lesssim 
h^{\b/(2s)-1},
$$
$$
\int_{B_1\setminus B_{h^{1/(2s)}}}\frac{|z|^\b}{h^{\frac{n+2s}{2s}}+|z|^{n+2s}}\,dz \leq \int_{B_1\setminus B_{h^{1/(2s)}}}|z|^{\b-n-2s} \,dz \lesssim 1+ h^{\b/(2s)-1},
$$
$$
\int_{\R^n\setminus B_1}\frac{1}{h^{\frac{n+2s}{2s}}+|z|^{n+2s}}\,dz \leq \int_{\R^n\setminus B_1}|z|^{-n-2s} \,dz \lesssim 1.
$$
To prove \eqref{estimate2}, we observe that the bound is trivial if $h \geq 1$, since
$\frac{1}{(t-\tau)^{\frac{n+2s}{2s}}+h^{n+2s}} \lesssim \frac{1}{h^{n+2s}}$ (recall that $|t-\tau| \leq T \lesssim 1$). On the other hand, if $h \in (0,1]$,
it suffices to split the integral over $[0,t-h^{2s}]$ and $[t-h^{2s},t]$, and argue as above.\\

\noindent$\bullet$ \textit{Proof of \eqref{eq:infty regularity holder 1}.}
Let us observe that, for $u<t$,
\begin{align*}
\bigl|\Df v(t,x)&-\Df v(\tau,x)\bigr|\\
&=\biggl|\int_0^t \int_{\R^n}\partial_t \G_s(t-\tau,x-y) [f(\tau,y)-f(\tau,x)]\,dy\,d\tau\\
&\qquad- \int_0^u \int_{\R^n}\partial_t \G_s(u-\tau,x-y) [f(\tau,y)-f(\tau,x)]\,dy\,d\tau\biggr|\\
&\lesssim \int_u^t \int_{\R^n}\bigl|\partial_t \G_s(t-\tau,x-y)\bigr||f(\tau,y)-f(\tau,x)|\,dy\,d\tau\\
&\qquad+\int_0^u \int_{\R^n}\Bigl|\partial_t \G_s(u-\tau,x-y)-\partial_t \G_s(t-\tau,x-y)\Bigr| \min\{|x-y|^\b,1\}\,dy\,d\tau\\
&=(T1)+(T2).
\end{align*}
Now, using \eqref{eq:fund sol 1}, by \eqref{estimate1} applied with $h=(t-\tau)$ we get that (T1) is bounded by
\begin{align*}
\int_u^t \int_{\R^n}\frac{\min\{|x-y|^\b,1\}}{(t-\tau)^{\frac{n+2s}{2s}}+|x-y|^{n+2s}}\,dy\,d\tau&
\leq \int_u^t \Bigl(1+ (t-\tau)^{\b/(2s)-1}\Bigr)\,d\tau\\
&\lesssim (t-u)+(t-u)^{\b/(2s)}.
\end{align*}
Concerning (T2), thanks to \eqref{eq:fund sol 1}, \eqref{eq:fund sol 2}, and \eqref{estimate1} with $h=(t-\tau)$, we can control it by
\begin{align*}
&\int_0^u \int_{\R^n} \min\left\{\frac{t-u}{u-\tau},1\right\}\frac{1}{((u-\tau)^{\frac{n+2s}{2s}}+|x-y|^{n+2s})}\min\{|x-y|^\b,1\}\,dy\,d\tau\\
&\lesssim \int_0^u \min\left\{\frac{t-u}{u-\tau},1\right\} \int_{\R^n} \frac{\min\{|x-y|^\b,1\}}{(u-\tau)^{\frac{n+2s}{2s}}+|x-y|^{n+2s}}\,dy\,d\tau\\
&\lesssim \int_0^{u-(t-u)} \frac{t-u}{u-\tau}\Bigl(1+ (u-\tau)^{\b/(2s)-1}\Bigr) \,d\tau+\int_{u-(t-u)}^u \Bigl(1+ (u-\tau)^{\b/(2s)-1}\Bigr) \,d\tau\\
&\lesssim (t-u)+(t-u)|\log(t-u)|+ (t-u)^{\b/(2s)},
\end{align*}
which proves the time  regularity of $\Df v$.

\noindent$\bullet$ \textit{Proof of \eqref{eq:infty regularity holder 2}.}
The proof of the spatial regularity is analogous: we write
\begin{align*}
&\Df v(t,x)-\Df v(t,z)\\
&=\int_0^t \int_{\R^n}\Bigl(\partial_t \G_s(t-\tau,x-y) [f(\tau,y)-f(\tau,x)]-\partial_t \G_s(t-\tau,z-y) [f(\tau,y)-f(\tau,z)] \Bigr)\,dy\,d\tau.
\end{align*}
Then, we split the spatial integral over two sets: the region where $\bigl\{|x-z| \leq |x-y|/2\bigr\}$
and the region where $\bigl\{|x-z| \geq |x-y|/2\bigr\}$.

On the first set, since  $\frac{1}{(t-\tau)^\frac{n+2s}{2s}+|x-y|^{n+2s}}$ and $\frac{1}{(t-\tau)^\frac{n+2s}{2s}+|z-y|^{n+2s}}$ are comparable, we can estimate the integrand by
$$
|f(\tau,y)-f(\tau,x)|\Bigl|\partial_t \G_s(t-\tau,x-y)-\partial_t \G_s(t-\tau,z-y)\Bigr|
+ |f(\tau,x)-f(\tau,z)|\Bigl|\partial_t \G_s(t-\tau,x-y)\Bigr|,
$$
which thanks to \eqref{eq:fund sol 1} and \eqref{eq:fund sol 2} can be bounded
by
$$
|y-x|^\beta\frac{|x-z|}{|x-y|}\frac{ 1 }{(t-\tau)^{\frac{n+2s}{2s}}+|x-y|^{n+2s}}
+|x-z|^\b \frac{1}{(t-\tau)^{\frac{n+2s}{2s}}+|x-y|^{n+2s}}.
$$
So, using \eqref{estimate2} with $h=|x-y|$ we get
\begin{align*}
&\int_0^t \int_{\{|x-z| \leq |x-y|/2\}}\Bigl|\partial_t \G_s(t-\tau,x-y) [f(\tau,y)-f(\tau,x)]-\partial_t \G_s(t-\tau,z-y) [f(\tau,y)-f(\tau,z)] \Bigr|\,dy\,d\tau\\
& \lesssim \int_{\{|x-z| \leq |x-y|/2\}} \left(\frac{|x-z|}{|x-y|^{1-\b}}+|x-z|^\b\right)\int_0^t\frac{1}{(t-\tau)^{\frac{n+2s}{2s}}+|x-y|^{n+2s}}\,d\tau \,dy\\
&\lesssim \int_{\{|x-z| \leq |x-y|/2\}} |x-z|^\b\min\left\{\frac{1}{|x-y|^n},\frac{1}{|x-y|^{n+2s}}\right\} \,dy\\
& \leq |x-z|^{\b}  \int_{\{|x-z| \leq |x-y|/2 \leq 1\}} |x-y|^{-n} + |x-z|^\b \int_{\{|x-y|/2 \geq 1\}}|x-y|^{-n-2s}\\
&\lesssim |x-z|^\b\bigl|\log|x-z|\bigr| \lesssim |x-z|^{\b-\e}\qquad\forall \,\e>0.
\end{align*}
Concerning the integral over the second set, we simply use \eqref{eq:fund sol 1} to bound the integrand by
$$
\frac{|x-y|^\beta}{(t-\tau)^\frac{n+2s}{2s}+|x-y|^{n+2s}}+ \frac{|x-z|^\beta}{(t-\tau)^\frac{n+2s}{2s}+|x-z|^{n+2s}}
$$
and observing that $\{|x-z| \geq |x-y|/2\} \subset B_{3|x-z|}(x)\cap B_{3|x-z|}(z)$ we get
\begin{align*}
&\int_0^t \int_{\{|x-z| \geq |x-y|/2\}}\Bigl|\partial_t \G_s(t-\tau,x-y) [f(\tau,y)-f(\tau,x)]-\partial_t \G_s(t-\tau,z-y) [f(\tau,y)-f(\tau,z)] \Bigr|\,dy\,d\tau\\
&\lesssim \int_0^t \int_{B_{3|x-z|}(x)}  \frac{|x-y|^\beta}{(t-\tau)^\frac{n+2s}{2s}+|x-y|^{n+2s}}\,dy\,d\tau\\
&=\int_{B_{3|x-z|}(x)}  |x-y|^\beta\int_0^t \frac{1}{(t-\tau)^\frac{n+2s}{2s}+|x-y|^{n+2s}}\,d\tau\,dy\\
&\lesssim \int_{B_{3|x-z|}(x)}  |x-y|^{\beta-n}\,dy \lesssim |x-z|^\b,
\end{align*}
where for the last but one inequality we used again \eqref{estimate2} with $h=|x-y|$.
This concludes the proof of \eqref{eq:infty regularity holder 2}.

\end{document}